\documentclass[oneside,11pt]{article}

\usepackage{amsmath}
\usepackage{amssymb}
\usepackage{pxfonts}
\usepackage{dsfont}
\usepackage{graphicx}
\usepackage{eucal}
\usepackage{mathrsfs}
\usepackage{theorem}
\usepackage{pifont}
 \usepackage{sectsty}
\usepackage{amscd}
\usepackage{color}
\usepackage{fancyhdr}
\usepackage{framed}
\usepackage[all]{xy}
\usepackage[backref=page]{hyperref}

\definecolor{shadecolor}{rgb}{0.8,0.8,0.8}

\theoremheaderfont{\fontfamily{pzc}\bfseries\large}
\newtheorem{theorem}{Theorem}[section]

\newtheorem{lemma}[theorem]{Lemma}
\newtheorem{proposition}[theorem]{Proposition}

\newtheorem{corollary}[theorem]{Corollary}
\newtheorem{definition}[theorem]{Definition}

\theorembodyfont{\rmfamily}

\newcommand{\specexercise}[1]{}

\newenvironment{proof}{{\flushleft \emph{Proof}:}}{\hfill\ding{110}}

\newenvironment{remark}{{\flushleft \fontfamily{pzc}\bfseries\large Remark:}}{}

\newcommand{\g}{\mathfrak{g}}

\newcommand{\Vol}{\text{Vol}}

\newcommand{\M}{\mathcal{M}}
\newcommand{\R}{\mathbb{R}}
\newcommand{\N}{\mathbb{N}}
\newcommand{\Z}{\mathbb{Z}}

\newcommand{\dist}{\operatorname{dist}}

\newcommand{\id}{\operatorname{Id}}
\newcommand{\supp}{\operatorname{supp}}
\newcommand{\Diff}{\operatorname{Diff}}
\newcommand{\Diffc}{\operatorname{Diff}_\text{c}}

\newcommand{\tzeta}{\tilde{\zeta}}
\newcommand{\ttheta}{\tilde{\theta}}
\newcommand{\tTheta}{\tilde{\Theta}}

\newcommand{\pl}{\partial}
\newcommand{\ind}{\mathds{1}}

\newcommand{\beq}{\begin{equation}}
\newcommand{\eeq}{\end{equation}}

\newcommand{\brk}[1]{\left(#1\right)}          
\newcommand{\Brk}[1]{\left[#1\right]}          
\newcommand{\BRK}[1]{\left\{#1\right\}}        
\newcommand{\Abs}[1]{\left|#1\right|}        

\numberwithin{equation}{section}

\begin{document}

\title{Vanishing geodesic distance for right-invariant Sobolev metrics on diffeomorphism groups}
\author{Robert L.~Jerrard\footnote{Department of Mathematics, University of Toronto.} \,and Cy Maor\footnotemark[1]}
\date{}
\maketitle

\begin{abstract}
We study the geodesic distance induced by right-invariant metrics on the group $\Diffc(\M)$ of compactly supported diffeomorphisms, for various Sobolev norms $W^{s,p}$.
Our main result is that the geodesic distance vanishes identically on every connected component whenever $s<\min\{n/p,1\}$, where $n$ is the dimension of $\M$.
We also show that previous results imply that whenever $s > n/p$ or $s \ge 1$, the geodesic distance is always positive.
In particular, when $n\ge 2$, the geodesic distance vanishes if and only if $s<1$ in the Riemannian case $p=2$, contrary to a conjecture made in Bauer et al.~\cite{BBHM13}.
\end{abstract}

\tableofcontents

\section{Introduction}
In this paper we mostly resolve a question 
about the geometry of the group $\Diffc(\M)$ of compactly supported diffeomorphisms of a Riemannian manifold $\M$, endowed with a right-invariant Sobolev metric;
see Section \ref{sec_2} below for the precise definition, as well as assumptions on $\M$.
Sobolev metrics on $\Diffc(\M)$ arise
in a variety of contexts. In particular, such a metric
turns $\Diffc(\M)$ into an infinite-dimensional Riemannian manifold, 
and a number of partial differential equations
relevant to fluid dynamics can be formulated as  geodesic flow in manifolds of this sort.
Sobolev metrics on $\Diffc(\M)$ are also relevant to the study of what are known as {\em shape spaces},
a concept with connections to areas such as computer vision and computational anatomy. We refer to \cite{BBM14} for a discussion of these and other sources of motivation.

A metric on $\Diffc(\M)$ gives rise to a notion of the length of  a path, and the induced
geodesic distance between a pair of elements  
is obtained by taking the infimum
of the lengths of all paths connecting the two diffeomorphisms.
If the metric is induced by the  $H^s$ Sobolev inner product for $s$ small enough,
the geodesic distance may vanish in the strong sense that
any two diffeomorphisms that can be connected
by a path can in fact be connected by a path of arbitrarily small length.
For large enough $s$, by contrast, the geodesic distance between any
two distinct diffeomorphisms is positive. Our aim is to identify the
precise threshold
that separates these two cases.

This question grows out of work of \cite{MM05}, who proved  (among other results) that the $H^s$ geodesic distance vanishes when $s=0$ and is positive when $s=1$. These results were extended to certain $s\in (0,1)$ by \cite{BBHM13,BBM13}, who proved that  for $\M$ of bounded geometry, the $H^s$ geodesic distance vanishes if $s<1/2$. They also proved that for one-dimensional manifolds, the geodesic distance
is positive when $s>1/2$, and for $\M =  \mathbb S^1$, it vanishes in the borderline case $s=\frac 12$.\footnote{Very shortly after we completed this manuscript, a proof that $H^{1/2}$ geoedesic distance vanishes for all one-dimensional manifolds was posted, see \cite{BHP18}.} 
Motivated by these facts, they conjectured that for arbitrary manifolds, the
induced $H^s$ geodesic distance should vanish if and only if $s\le 1/2$.

It turns out to be illuminating to embed this conjecture in a larger family of questions,
about the vanishing of the geodesic distance induced
by right-invariant fractional Sobolev norms $W^{s,p}$, for $1\le p < \infty$, see again Section \ref{sec_2} for details (note that we do not consider the case $p=\infty$ in this paper unless explicitly noted).
The arguments used by \cite[Theorem~5.7]{MM05}, \cite[Theorem~4.1]{BBHM13} then imply the following:
\begin{theorem}[\cite{MM05,BBHM13}]
The induced $W^{s,p}$-distance is positive whenever
$sp>n$ or $s\ge 1$.
\end{theorem}

Our main result shows that these results
are essentially sharp: 
\begin{theorem}
The induced $W^{s,p}$-distance is vanishes whenever
$sp<n$ and $s< 1$.
\end{theorem}
These results are stated in a more detailed way in Theorem~\ref{main_thm}.
In particular, contrary to the conjecture of \cite{BBHM13},  
we have the following corollary:
\begin{corollary}
If $\M$ is a manifold of dimension at least $2$, then the $H^s$ geodesic distance vanishes if and only if $s<1$.
\end{corollary}

We conclude this informal introduction by describing some ingredients in our analysis.
First, we remark that the positivity proof of \cite[Theorem~5.7]{MM05} can be understood to show that for any $s\ge 0$, paths in $\Diffc(\M)$ of short length
must involve compression of (parts of) the support of the diffeomorphism into
very small sets, and that this compression can always be detected by $W^{s,p}$-norms
when $s\ge 1$.
The positivity proof of \cite[Theorem~4.1]{BBHM13} relies on the observation that 
any motion, no matter how small its support, can always be detected by
any $W^{s,p}$-norm that embeds into $L^\infty$.
This property holds whenever $sp>n$.

If $s<1$, it turns out that one can compress parts of the manifold into arbitrarily small regions, for arbitrarily small cost; and if $sp<n$ 
one can transport small regions of the manifold for a long distance with small cost. Therefore, if $s<\min \{n/p,1\}$, one might expect  the geodesic distance to vanish. 
Our proof that this is indeed the case has two main points. The first is to devise a strategy for alternating compression and transport of small sets in order to flow the identity mapping, say, onto a fixed target diffeomorphism at low cost. The second point is that the transport step requires some care in order to arrive at (or sufficiently close to) a fixed target, while still remaining small in the relevant norms. 
We achieve this
by first constructing a flow, relying in part on ideas of \cite{BBHM13}, that exactly reaches the desired target;
however in order for this flow to be in the right Sobolev space we need to regularize it.
This regularization, and the error controlling that follows it, form the majority of the technical part of this paper.

Our heuristic arguments, described above, for vanishing geodesic distance apply also in the
endpoint case $s=\frac n p <1$, since $W^{n/p,p}$ also fails to embed into $L^\infty$ in this case.
As mentioned above, it is known that the $W^{1/2,2}$-induced geodesic distance
vanishes on $\Diffc(\mathbb S^1)$,
and although we do not present the details, the proof of
\cite{BBHM13} can be readily extended to $W^{1/p,p}$ for all 
$1<p<\infty$. In general, however, although it is natural to conjecture that
the $W^{n/p,p}$-induced geodesic distance vanishes on $n$ dimensional manifolds when $p>n$,
the critical scaling makes constructions
delicate, and this question remains open for $\dim\M > 1$.

\section{Preliminaries and main result}\label{sec_2}

Let $(\M,\g)$ be a Riemannian manifold of \emph{bounded geometry}, that is $(\M,\g)$ has a positive injectivity radius and all the covariant derivatives of the curvature are bounded: $\|\nabla^i R\|_\g < C_i$ for $i\ge 0$.
We denote by $\Gamma_c(T\M)$ the Lie-algebra of compactly supported vector fields on $\M$, 
and by $\Diffc(\M)$ the group of compactly supported diffeomorphisms of $\M$, that is the diffeomorphisms $\phi$ for which the closure of $\{\phi(x)\ne x\}$ is compact.

A smooth path $\{\phi_t \}_{ t\in [0,1]}$ in $\Diffc(\M)$ 
can be described in terms of the velocity vector fields $\{u(t,\cdot)\}_{t\in[0,1]}$ such that 
$\pl_t \phi_t =  u(t, \phi_t)$ for $0\le t\le 1$. 
Given $\{\phi_t\}$, we find $u$ by setting  $u(t, \cdot) := \partial_t \phi_t \circ \phi_t^{-1}$, and conversely $\{\phi_t\}_{t\in [0,1]}$ 
may be recovered from $u$ and $\phi_0$ by standard ODE theory. 
Given a norm $\|\cdot\|_{A}$ on $\Gamma_c(T\M)$ we can then define the \textbf{geodesic distance} between $\phi_0,\phi_1\in \Diffc(\M)$ by
\[
\dist_A(\phi_0,\phi_1) := \inf\BRK{\int_0^1 \|u(t) \|_{A} \,dt \,\,:\,\,
\mbox{ $\pl_t \phi_t =  u(t, \phi_t)$ for $0\le t\le 1$ }}.
\]
Note that $\dist_A$ forms a semi-metric on $\Diffc(\M)$, that is it satisfies the triangle inequality but may fail to be positive.

This is the geodesic distance of the \textbf{right-invariant Finsler metric on $\Diffc(\M)$} induced by $\|\cdot\|_{A}$, which is defined as
\[
\|X\|_{\phi,A} := \|X\circ \phi^{-1}\|_A
\]
for every $\phi\in \Diffc(\M)$ and $X\in T_\phi \Diffc(\M)$.
If $\|\cdot\|_A$ comes from an inner-product, it defines a Riemannian metric on $\Diffc(\M)$ in a similar manner.
See \cite{BBHM13} for more details.
The right-invariance of $\dist_A$ is summarized in the following lemma:

\begin{lemma}[Right-invariance]
\label{lm:right_invariance}
For $\psi,\phi_0,\phi_1\in \Diffc(\M)$, we have 
\[
\dist_A(\phi_0 \circ \psi, \phi_1 \circ \psi) = \dist_A(\phi_0,\phi_1).
\]
In particular,
\[
\dist_A(\id,\psi) = \dist_A(\id,\psi^{-1}),
\]
and 
\[
\dist_A(\id,\phi_1 \circ \phi_0) \le \dist_A(\id, \phi_1) + \dist_A(\id, \phi_0).
\]
\end{lemma}

\begin{proof}
Let $t\mapsto \phi_t\in \Diffc(\M)$ be a curve from $\phi_0$ to $\phi_1$.
Denote $u_t = \pl_t \phi_t \circ \phi_t^{-1}$.
Define $\Phi_t = \phi_t\circ \psi$. This is a curve from $\phi_0\circ \psi$ to $\phi_1\circ \psi$.
We then have
\[
\pl_t \Phi_t = \pl_t \phi_t \circ \psi = \pl_t \phi_t \circ \phi_t^{-1} \circ \Phi_t = u_t \circ \Phi_t,
\]
from which the first claim follows immediately.
The second and third claims follow from the first, since
\[
\dist_A(\id,\psi^{-1}) = \dist_A(\psi\circ \psi^{-1},\psi^{-1}) = \dist_A(\psi,\id),
\]
and
\[
\dist_A(\id,\phi_1 \circ \phi_0) \le \dist_A(\id, \phi_0) + \dist_A(\phi_0,\phi_1 \circ \phi_0) = \dist_A(\id, \phi_0) + \dist_A(\id,\phi_1).
\]
\end{proof}

We are interested in  fractional Sobolev $W^{s,p}$-norms, and in particular in $H^s := W^{s,2}$, for $s\in(0,1)$.  We adopt the following as our basic definition, from among a number
of equivalent formulations.

\begin{definition}
\label{def:fractional_Sobolev}
For $0<s<1$ and $1\le p<\infty$, 
the $W^{s,p}$-norm of a function $f\in L^p(\R^n)$ is given by
\[
\|f\|_{s,p}^p = \| f\|_{L^p}^p + \int_{\R^n}\int_{\R^n} \frac {|f(x)-f(y)|^p}{|x-y|^{n+sp}}\, dx\,dy .
\]
\end{definition}

Given a Riemannian manifold $(\M,\g)$ of bounded geometry, this norm can be extended to $\Gamma_c(T\M)$ using trivialization by normal coordinate patches on $\M$ (see \cite[Section~2.2]{BBM13} for details).
We will denote the induced geodesic distance on $\Diff_c(\M)$ by $\dist_{s,p}$. 
When $p=2$, we will denote $\dist_{s,2}$ by $\dist_s$ for simplicity.
Different choices of charts result in equivalent metrics, and therefore the question of vanishing geodesic distance is independent of these choices.

Instead of using Definition~\ref{def:fractional_Sobolev} directly, we will bound the $W^{s,p}$-norm using an interpolation inequality:

\begin{proposition}[fractional Gagliardo-Nirenberg interpolation inequality]
\label{pn:GN_inequality}
Assume that $1<p<\infty$. For every $f\in W^{1,p}(\R^n)$ and $s\in (0,1)$,
\[
\| f\|_{s,p} \le C_{s,p} \|f\|_{L^p}^{1-s} \|f\|_{1,p}^s\, , \qquad\mbox{ where }\ \ 
\|f\|_{1,p}^p := \|f\|_{L^p}^p+ \|df\|_{L^p}^p.
\]

\end{proposition}
For a proof, see  for example \cite[Corollary~3.2]{BM01}.  In fact this is the only property of the $W^{s,p}$-norm that we will use. We remark that when $p=2$, the above inequality
(with $C=1$) follows immediately from H\"older's inequality, if one uses the equivalent norm $\| f\|_{s,2}^2 = \int_{\R^n}(1+|\xi|^2)^{s/2}|\hat f(\xi)|^2 d\xi$, where $\hat f$ denotes
the Fourier transform.

The main result of this paper is the following.

\begin{theorem}\label{main_thm}
Let $(\M,\g)$ be an $n$-dimensional Riemannian manifold of bounded geometry.
\begin{enumerate}
\item
If $p\in [1,\infty)$ and $s< \min\{ 1, n/p\}$, then  $\dist_{s,p}(\phi_0,\phi_1)= 0$ whenever $\phi_0,\phi_1$
belong to the  same path-connected component of $\Diffc(\M)$.
\item
If $s\ge 1$ or $sp>n$ then  $\dist_{s,p}(\phi_0,\phi_1)>0$ for any two distinct $\phi_0,\phi_1\in \Diffc(\M)$.
\end{enumerate}
\end{theorem}

The second assertion is a direct consequence of known arguments in the case $p=2$.
So is the first one for the case $n=1$.
The new point is the vanishing of geodesic distance for all $ s <\min \{1, n/p\}$
whenever $n\ge 2$. 

Note that Proposition~\ref{pn:GN_inequality}, which is used extensively in the proof of the first part of Theorem~\ref{main_thm}, does not hold for $p=1$. However, Theorem~\ref{main_thm} does hold in this case as well; as explained in more detailed in Section~\ref{sec:Wsp}, our proof for vanishing $W^{s,p}$-distance for $p$ close enough to $1$ implies vanishing $W^{s,1}$-distance.

In the remainder of this section we quickly verify that known results about the case $p=2$ extend to the more general setting we consider here, and we present the
reduction, also well-known in the $H^s$ case, that will allow us to complete the proof of the theorem by 
showing that $\dist_{s,p}(\id,\Phi)=0$ for a single compactly supported diffeomorphism on $\R^n$.

\paragraph{Positive geodesic distance}
First, assume that  $\phi_0,\phi_1$ are two distinct elements of $\Diffc(\M)$, and let $u$ be any time-dependent vector field generating a path
$\phi:[0,1]\to \Diffc(\M)$  connecting $\phi_0$ to $\phi_1$, via the ODE
$\partial_t\phi_t = u(t,\phi_t), 0<t <1$.
The proof of \cite[Theorem~5.7]{MM05} uses a clever integration by parts to show that for any $\rho, \zeta\in C^1_c(\M)$, 
\[
\left|
\int_\M \rho (\zeta\circ \psi_1 - \zeta) \mbox{vol}(g)
\right| = \left|\int_0^1 \int_\M (\zeta\circ \psi_t) \mbox{div}(\rho u_t) \mbox{vol}(g)\,dt\right|, \qquad
\psi_t := \phi_0\circ \phi_t^{-1}.
\]
By a suitable choice of $\rho, \zeta$, this implies that  $0<c\le C\int_0^1 \|u(t)\|_{1,p} dt$ for $p\ge 1$, where the constants depend on $\phi_1,\phi_2,\rho, \zeta, p$. This shows the positivity of the geodesic distance in $W^{1,p}$ for any $p\ge 1$, and hence
(since these spaces embed into $W^{1,p}$) in $W^{s,p}$ for $s\ge 1$.

On the other hand, if $s>n/p$, then $W^{s,p}$ embeds into some $C^{0,\alpha}$ (see for example \cite[Theorem~8.2]{DPV12}) and hence into $L^\infty$. Thus 
$\| \partial_t\phi_t\|_{L^\infty} = \| u(t) \|_{L^\infty} \le C \|u(t)\|_{s,p}$, and as
noted in \cite[Theorem~4.1]{BBHM13}, the positivity of $\dist_{s,p}$ follows directly:
\[
|\phi_1(x)-\phi_0(x)|  =
 \left|\int_0^1 \partial_t\phi_t(x)\, dt\right|
\le
C\int_0^1 \| u(t)\|_{s,p} dt \qquad\mbox{ for every }x\in \M.
\]
Note that it also follows that the geodesic distance is positive for $L^\infty=W^{0,\infty}$.

For $sp<n=1$,
the proof of vanishing geodesic distance in \cite{BBHM13} in the case $p=2$ relies on an explicit construction (incorporated into \eqref{eq:ttheta_def} below) of a transportation scheme of the identity to a single diffeomorphism, that has arbitrarily small cost;
this arbitrarily small cost follows from the fact that the $W^{s,p}$-norm of
the characteristic function of an interval tends to zero with the length of the interval.
For general $sp<n=1$, this is well-known and can easily be verified  from 
Definition \ref{def:fractional_Sobolev}. Once this is noted, the proof goes through
with no change.

\paragraph{Reduction to a single diffeomorphism}
The following proposition states an important property of $(\Diffc(\M),\dist_{s,p})$ --- it is either a metric space, or it collapses completely, that is, the geodesic distance in any connected component of $\Diffc(\M)$ vanishes.
In other words, if $(\Diffc(\M),\dist_{s,p})$ is not a metric space, then any two diffeomorphisms in the same connected component can be connected by a path of arbitrary short $W^{s,p}$-length.
\begin{proposition}
\label{pn:normal_subgroup}
Denote by $\Diff_0(\M)$ the connected component of the identity (all diffeomorphisms in $\Diffc(\M)$ for which there exists a curve between them and $\id$). 
\begin{enumerate}
\item $\Diff_0(\M)$ is a simple group.
\item $\BRK{\phi : \dist_{s,p}(\id,\phi) = 0}$ is a normal subgroup of $\Diff_0(\M)$.
	Therefore, it is either $\BRK{\id}$ or the whole $\Diff_0(\M)$.
\end{enumerate}
\end{proposition}

This is proved in  \cite[p.~15]{BBHM13} (see also \cite[Lemma~7.10]{BBM14})
when $p=2$, and the proof goes through with essentially no change in
our setting. We recall the idea. The first conclusion is classical (and is independent of the norm). To establish the second,
we consider $\phi, \psi\in \Diffc(\M)$ such that $\dist_{s,p}(\id, \phi)= 0$, and we
must show that $\dist_{s,p}(\id,   \Phi)=0$ for $\Phi :=\psi^{-1}\circ \phi \circ \psi$. 
To do this,  note that if $\phi_t$, $0\le t\le 1$ is a path connecting $\id$ to $\phi$, then
$\Phi_t := \psi^{-1}\circ \phi_t\circ\psi$ connects $\id$ to $\Phi$. The conclusion thus follows
by verifying that $\int_0^1 \| \partial_t\Phi_t\circ \Phi^{-1}\|_{s,p}dt\le C\int_0^1 \| \partial_t\phi_t\circ \phi^{-1}\|_{s,p} \, dt $, where $C$ may depend on $\psi, (\M,\g), s,p$ but not $\phi$. In fact a pointwise inequality of the integrands holds for every $t$. This follows after a computation from the fact that for $h\in C^\infty(M)$ and $\psi\in \Diffc(\M)$, the operations of pointwise multiplication $u\mapsto h \cdot u$ and
composition $u\mapsto u\circ\psi$ are bounded linear operators on $W^{s.p}(\M)$,
see Theorems~4.2.2 and 4.3.2 in \cite{Tri92}.

\paragraph{The strategy for proving vanishing geodesic distance} 
The proof of part 1 of Theorem~\ref{main_thm} for $n\ge 2$ goes as follows:
\begin{enumerate}
\item For $sp<n$ and $n\ge 2$, we will show that there exists at least one
nontrivial $\Phi\in \Diffc(\R^n)$ such that $\dist_{s,p}(\id,\Phi) = 0$.
\item For general $(\M,\g)$ of bounded geometry, 
we can push-forward this example in $\R^n$
to obtain a diffeomorphism $\widetilde \Phi$, supported in a single coordinate
chart used in the definition of induced $W^{s,p}$ geodesic distance.
Then the definitions imply that $\dist_{s,p}(\id, \widetilde \Phi) = 0$.
(see \cite{BBM13} for a similar argument).
\item Part 1 of Theorem~\ref{main_thm} then follows from Proposition~\ref{pn:normal_subgroup}. 
\end{enumerate}

In the rest of the paper we treat the first point.
For simplicity, we first consider the special case $p=2, \M = \R^2$, and
we show that  $\dist_{s}(\id, \Phi) := \dist_{s,2}(\id, \Phi) = 0$ for a particular $\Phi\in \Diffc(\R^2) $.
This construction, carried out in Section \ref{sec:2dc}, contains all the ingredients of more general cases. In Section \ref{sec:HD} we present a much simpler construction that works
when $p=2, s<1$ and $n\ge 3$. 
Finally, in Section \ref{sec:Wsp} we show how
to modify these arguments to complete the proof of the theorem in the general case.

\section{Two-dimensional construction}\label{sec:2dc}
In this section we prove the following:
\begin{theorem}
\label{thm:main_2D}
Let $\zeta \in C_c^\infty((0,1)^2)$ satisfying $\zeta \ge 0$, $\pl_1\zeta > -1$. 
Denote $\phi(x,y) = x + \zeta(x,y)$,
and define $\Phi\in \Diffc(\R^2)$ by $\Phi(x,y) = (\phi(x,y),y)$.
Then $\dist_s(\Phi,\id) = 0$ for every $s\in[0,1)$.
\end{theorem}

We start with a general outline and heuristics of the proof.
Fix $k\in \N$.
In Section~\ref{sec:Step_I_2D} we decompose $\Phi$ as follows:
\[
\Phi = \Phi_2\circ \Phi_1, \qquad \Phi_i = (\phi_i(x,y),y) = (x+\zeta_i(x,y),y) \in \Diffc(\R^2),
\]
where $\zeta_i$ is supported on the union of $\approx k$ strips $(0,1)\times I_j$, $|I_j| \approx k^{-1}$.
In Sections~\ref{sec:Step_II_2D}--\ref{sec:Step_IV_2D}, we show that $\dist_s(\Phi_1,\id) = o(1)$, when $k\to \infty$; the proof for $\Phi_2$ is analogous, and since $k$ is arbitrary, the conclusion $\dist_s(\Phi,\id) = 0$ follows by Lemma~\ref{lm:right_invariance}.

In order to prove $\dist_s(\Phi_1,\id) = o(1)$, we decompose $\Phi_1$ as follows: 
\[
\Phi_1 = \Gamma^{-1} \circ {\Psi}^{-1} \circ \Theta \circ \Psi, \qquad \Gamma,\Theta,\Psi \in \Diffc(\R^2),
\]
where 
\begin{enumerate}
\item $\Psi(x,y) = (x,\psi(x,y))$ squeezes the intervals $I_j$ into intervals of length $\approx \lambda$ for $\lambda$ of the form $\lambda = e^{-\alpha} k^{-1}$,
where $\alpha = \alpha(k)$ is a (moderately large) parameter, to be determined.
	In Section~\ref{sec:Step_II_2D} we define $\Psi$ and show that $\dist_s(\Psi,\id) \lesssim \alpha k^{-(1-s)}$.
	
	This stage compresses the support of $\Phi_1$ into small
	sets that can then, in the next stage, be transported large distances at low cost, owing to the subcriticality of $H^s(\R^2)$ for $s<1$. This concentration can be
	achieved at low cost (for $s<1$) because no point is moved very far. This requires
	the striped nature of the support of $\Phi_1$, and it is the reason for
	the decomposition $\Phi=\Phi_2\circ \Phi_1$.

\item $\Theta(x,y) = (\theta(x,y),y)$ maps $x$ almost to its right place, that is $\theta(x,\psi(x,y)) - \phi_1(x,y) \ll 1$.
	$\Theta$ is defined (as the endpoint of a given flow) via a construction similar to the construction (for $s<1/2$) in \cite{BBHM13,BBM13}; in order for it to work for $s\in [1/2,1)$, we need to regularize the flow (and therefore $\theta(x,\psi(x,y)) \ne \phi_1(x,y)$).
	We define $\Theta$ in Section~\ref{sec:Step_III_2D}, show that $\dist^2_s(\Theta,\id) \lesssim k\lambda^{2-s}\delta^{-s}$, where $\delta\ll \lambda$ is a regularization parameter to be determined.
	 The main part of this section consists of proving bounds on $\theta(x,\psi(x,y)) - \phi_1(x,y)$ and on the derivatives of $\theta$.
	 
	 The key idea in the construction of the flow is that at every given time its support is very small in both $x$ and $y$; the subcriticality of $H^s(\R^2)$ then implies that its $H^s$-norm at any given time is small.
	 For $H^s(\R^n)$, $n>2$ (and more generally, for $W^{s,p}(\R^n)$, $n>sp+1$), the squeezing in the $(n-1)$ $y$-directions done in the previous step is enough to guarantee a small $H^s$-norm of flows in the $x$ direction, that do not have small support in the $x$ direction (i.e.,~that the projection of the support on the $x$-axis is not small).
	 This is why in this case there in a much simpler construction in which the subtleties of this stage can be avoided.
	 
\item In Section~\ref{sec:Step_IV_2D} we show that the error $\Gamma = {\Psi}^{-1} \circ \Theta \circ \Psi \circ \Phi_1^{-1}$ satisfies $\dist_s(\Gamma,\id) \lesssim k^s\delta^{1-s}\lambda^{-(1-s)}$, by showing that the affine homotopy between $\id$ and $\Gamma$ is a path of small $H^s$-distance.
	This uses the bounds on $\theta$ from Section~\ref{sec:Step_III_2D}.
\end{enumerate}
Finally, we show that $\alpha$ and $\delta$ can be chosen such that, as $k\to \infty$,
\[
\dist_s(\Psi,\id) = o(1), \qquad \dist_s(\Theta,\id) = o(1), \quad \text{and} \quad \dist_s(\Gamma,\id) = o(1),
\]
and then $\dist_s(\Phi_1,\id) = o(1)$ follows from Lemma~\ref{lm:right_invariance}.

A short video presenting the main stages of the construction can be found in the following link: \href{http://www.math.toronto.edu/rjerrard/geo_dist_diffeo/vanishing.html}{www.math.toronto.edu{/}rjerrard{/}geo\_dist\_diffeo{/}vanishing.html}. 
The flow in the video involves no regularization in the construction of $\Theta$ (as it would not be visible in this resolution), and therefore the error-correction term $\Gamma$ is not needed, and $\Theta = \Psi \circ \Phi_1 \circ \Psi^{-1}$.
The video contains the following stages:
\begin{enumerate}
\item Compression of several disjoint intervals in the vertical direction (a path from $\id$ to $\Psi$). 
\item A flow in the horizontal direction, from $\Psi$ to $\Theta \circ \Psi = \Psi\circ \Phi_1$. Note that at any given time the flow is supported on a union of very small rectangles.
\item Undoing the squeezing stage, that is flowing from $\Psi \circ \Phi_1$ to $\Phi_1$.
\item Repeating steps 1--3 for $\Phi_2$, resulting in $\Phi_2\circ \Phi_1 = \Phi$.
\end{enumerate}

\begin{remark}
Throughout this paper, we use big $O$ and small $o$ notations with respect to the limit $k\to \infty$.
We will also use notations such as $|I_j| \approx k^{-1}$ above, meaning that there exist $c_2\ge c_1>0$ such that $c_1k^{-1} \le |I_j| \le c_2 k^{-1}$.
Finally, $a \lesssim b$, means $a \le Cb$ for some constant $C$ (that can depend on the dimension $n$ and the Sobolev exponent $s$).
\end{remark}

\subsection{Step I: Splitting into strips}
\label{sec:Step_I_2D}
Fix $k \in \N$.
Define the following subintervals of $(0,1)$:
\[
S_1^i := \Brk{\frac{8i-3}{k}, \frac{8i + 3}{k}}, \qquad 
L_1^i := \Brk{\frac{8i-2}{k}, \frac{8i + 2}{k}}, \qquad 
i\in \Z, 
\]
\[
S_2^i := \Brk{\frac{8i+1}{k}, \frac{8i + 7}{k}}, \qquad 
L_2^i := \Brk{\frac{8i+2}{k}, \frac{8i + 6}{k}}, \qquad 
i\in \Z, 
\]
and denote $S_j = \cup_i S_j^i\cap [0,1]$, $L_j = \cup_i L_j^i\cap[0,1]$.
Let $\chi:[-4,4]\to [0,1]$ be a smooth function satisfying $\supp \chi\subset (-3,3)$ and $\chi|_{[-2,2]} \equiv 1$.
Extend $\chi$ periodically, and define $\chi_k(y) = \chi(ky)$ on $(0,1)$.
Note that $\supp \chi_k \subset S_1$, $\chi_k|_{L_1} \equiv 1$, and $|\chi_k'| \lesssim k$.
See Figure~\ref{fig:Step_I}.

\begin{figure}
\begin{center}
\includegraphics[height=1.5in]{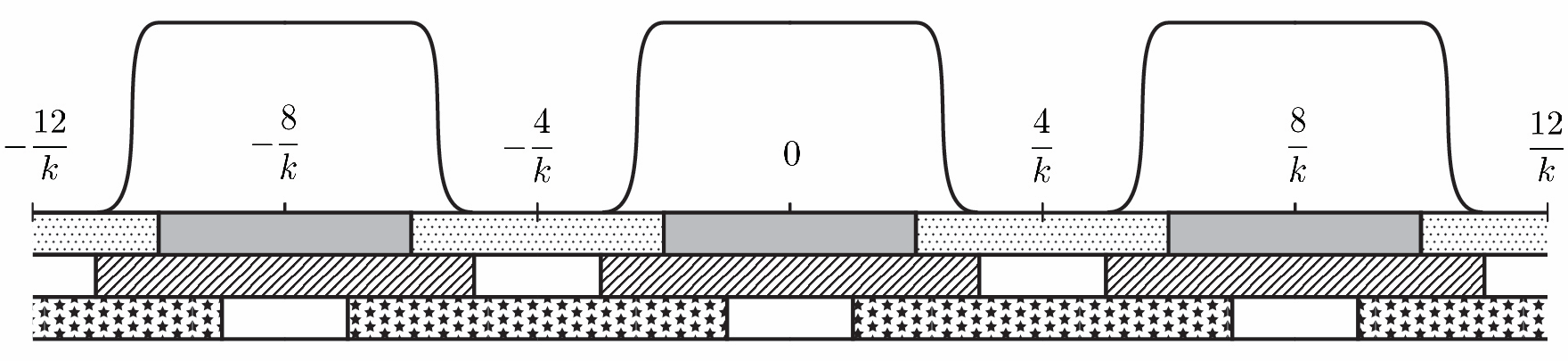} 
\end{center}
\caption{A sketch of $\chi_k$. The solid grey part of the top strip below the axis denotes $L_1$, where $\chi_k\equiv 1$; the dotted part of this strip denotes $L_2$.
The marked part of the middle strip denotes $S_1$, which contains $\supp(\chi_k)$, and hence $\supp(\zeta_1(x,\cdot))$.
The marked part of the bottom strip denotes $S_2$, which contains $\supp(\zeta_2(x,\cdot))$. }
\label{fig:Step_I}
\end{figure}

Define $\zeta_1(x,y) = \zeta(x,y) \chi_k(y)$. Note that
\beq
\label{eq:zeta_1_bounds_0}
\zeta_1|_{(0,1)\times L_1} = \zeta,
\eeq
\beq
\label{eq:zeta_1_bounds_1}
\supp(\zeta_1) \subset (0,1)\times S_1, 
\eeq
and
\beq
\label{eq:zeta_1_bounds_2}
0\le \zeta_1\le C,\quad
-1 + C^{-1}< \pl_x \zeta_1 < C, \quad
|\pl_y \zeta_1| < Ck,
\eeq
where $C$ is independent of $k$. The bounds \eqref{eq:zeta_1_bounds_2} follow from the bounds $0\le \zeta\le C$, $|d\zeta|<C$, $\pl_x \zeta > -1 + C^{-1}$ and $|\chi_k'| < Ck$.
Define 
\[
\Phi_1 = (\phi_1(x,y),y) = (x+\zeta_1(x,y),y), \qquad \Phi_2 = \Phi\circ \Phi_1^{-1} = (\phi_2(x,y),y).
\]
From \eqref{eq:zeta_1_bounds_0}--\eqref{eq:zeta_1_bounds_2}, it follows that
we can write $\phi_2(x,y) =  x+\zeta_2(x,y)$, with $\zeta_2$  satisfying the
bounds \eqref{eq:zeta_1_bounds_2}, and  property \eqref{eq:zeta_1_bounds_1} with $S_2$ in place of $S_1$.
Indeed, if $(x,y)\in (0,1)^2\setminus (0,1)\times S_2$, then $y\in L_1$, and hence, from \eqref{eq:zeta_1_bounds_0} it follows that $\phi_2(x,y) = x$, and therefore \eqref{eq:zeta_1_bounds_1} holds for $\zeta_2$ (with $S_1$ replaced by $S_2$).
Since $\zeta_1\le \zeta$ and $\zeta_2(\phi_1(x,y),y) = \zeta(x,y) - \zeta_1(x,y)$, it follows that $0\le \zeta_2 \le C$.
Finally, \eqref{eq:zeta_1_bounds_2} implies that $C^{-1} < \pl_x \phi_1 < C$ and $|\pl_y \phi_1| < Ck$; the inverse function theorem then implies the bounds \eqref{eq:zeta_1_bounds_2} for $\zeta_2$. 

In the rest of this section we are going to prove that $\dist_s(\Phi_1,\id) =o(1)$.
This relies only on properties \eqref{eq:zeta_1_bounds_1}--\eqref{eq:zeta_1_bounds_2}; hence, the result also applies to $\Phi_2$, since $\zeta_2$ satisfies the same assumptions.

\subsection{Step II: Squeezing the strips}
\label{sec:Step_II_2D}

\begin{lemma}
\label{lem:squeezing_2D}
Fix $\alpha \gg 1$. 
There exists a diffeomorphism $\Psi\in \Diffc(\R^2)$, $\Psi(x,y) = (x,\psi(x,y))$, such that
\beq
\label{eq:squeezing_2D}
\psi(x,y) = e^{-\alpha}\brk{y-\frac{8i}{k}} + \frac{8i}{k}, \qquad (x,y)\in [0,1]\times S_1^i\cap [0,1],
\eeq
and
\beq
\label{eq:squeezing_2D_H_s_dist}
\dist_s(\Psi,\id) \lesssim \alpha k^{-(1-s)}. 
\eeq
\end{lemma}

In other words, $\psi$ squeezes each intervals $S_1^i$ linearly around their midpoint by a factor of $e^{-\alpha}$, and has a small cost.

\begin{proof}
Let $u_1 \in C_c^\infty((-4,4))$, such that $u_1(y) = -y$ for $y\in [-3,3]$, and extend periodically. 
Let $\chi \in C_c^\infty(\R^2)$ such that $\chi\equiv 1$ on $[0,1]^2$.
Define $u_k(x,y) := \frac{\alpha}{k} u_1(ky)\chi(x,y)$.

Note that 
\[
\|u_k\|_{L^2} \lesssim \alpha/k, \qquad \|d u_k\|_{L^2} \lesssim \alpha.
\]
Therefore, by Proposition~\ref{pn:GN_inequality} we have
\beq
\label{eq:squeezing_2D_vec_field}
\|u_k\|_{H^s} \lesssim \frac{\alpha^{1-s}}{k^{1-s}} \alpha^s = \frac{\alpha}{k^{1-s}}.
\eeq

Let $\psi(t,x,y)$ be the solution of 
\[
\pl_t \psi = u_k(x,\psi), \qquad \psi(0,x,y) = y.
\]
Define $\psi(x,y) := \psi(1,x,y)$, and $\Psi(x,y) := (x,\psi(x,y))$.
A direct calculation shows that for $(x,y)\in [0,1]\times [-3/k, 3/k]$, $\psi(y) = y e^{-\alpha}$,
so by periodicity and the fact that $\chi\equiv 1$ on $[0,1]^2$, $\psi$ satisfies \eqref{eq:squeezing_2D}.
 
The trajectory from $\id$ to $\Psi$ defined by $\Psi_t(x,y) = (x,\psi(t,x,y))$, together with the bound \eqref{eq:squeezing_2D_vec_field}, implies \eqref{eq:squeezing_2D_H_s_dist}.
\end{proof}

Note that in $[0,1]^2$, $\psi$ is independent of $x$. Therefore, slightly abusing notation, we write
\[
\Psi(x,y) = (x,\psi(y)), \qquad \Psi^{-1}(x,y) = (x,\psi^{-1}(y)).
\]
We will later have $\alpha$ depend on $k$.
Since eventually we want $\dist_s(\Psi,\id) = o(1)$ when $k\to \infty$, \eqref{eq:squeezing_2D_H_s_dist} implies the bound
\beq
\label{eq:bounds_alpha}
\alpha \ll k^{1-s}.
\eeq

\subsection{Step III: Flowing along the squeezed strips}
\label{sec:Step_III_2D}

Denote 
\[
\lambda(\alpha,k) = \frac{e^{-\alpha}}{k},
\]
and consider
\[
\Phi_1 \circ \Psi^{-1} (x,y) = (x + \zeta_1(x,\psi^{-1}(y)), \psi^{-1}(y)) =: (x + \tzeta_1(x,y), \psi^{-1}(y)).
\]
Since $\zeta_1$ is supported inside $(0,1)\times S_1$, we have that $\tzeta_1 =\zeta_1\circ \Psi^{-1}$ is supported on $(0,1) \times \psi(S_1)$, that is, on $\approx k$ strips of thickness $\approx \lambda$.
Furthermore, from \eqref{eq:zeta_1_bounds_2} and \eqref{eq:squeezing_2D} we have
\beq
\label{eq:tilde_zeta_1_bounds}
\tzeta_1\ge 0,\quad
-1 + C^{-1}< \pl_x \tzeta_1 < C, \quad
|\pl_y \tzeta_1| < C\lambda^{-1}.
\eeq

We start by defining a path from $\id$ to
\[
\tTheta := \Psi \circ \Phi_1 \circ \Psi^{-1} (x,y) = (x + \tzeta_1(x,y), y),
\]
using a slight variation of the construction of \cite[Lemma 3.2]{BBHM13} that proves that the $H^s$ geodesic distance is vanishing for $s<1/2$.
Let
\beq
\label{eq:tau_g_def}
\tau_{y}(x) = x -\lambda \tzeta_1(x,y), \qquad g_{y}  = \tau_{y}^{-1}.
\eeq
It is clear that $\tau_{y}$ is increasing for all small enough $\lambda$.
We will henceforth restrict our attention to such $\lambda$, for which the definition of $g_y$ makes sense.
We will also write $\tau(x,y)$ and $g(t,y)$ instead of $\tau_{y}(x)$ and $g_{y} (t)$.
Define
\[
\tTheta(t,x,y) = (\ttheta(t,x,y), y) 
\]
by
\beq
\label{eq:ttheta_def}
\ttheta(t,x,y) := 
\begin{cases}
	x							&\mbox{ if }t\le  \tau(x,y) \\
	x+   (1+\lambda)^{-1}(t- \tau(x,y))		&\mbox{ if }\tau(x,y) \le t  \le x+\tzeta_1(x,y) \\
	x+\tzeta_1(x,y)						&\mbox{ if }x+\tzeta_1(x,y)\le t \le 1.
\end{cases}
\eeq

Note that $\ttheta$ solves
\[
\frac \pl{\pl t}\ttheta(t,x,y) = u(t,\ttheta(t,x,y),y), \qquad \ttheta(0,x)=x,
\]
where
\begin{equation}
\label{eq:def_u}
u_t(x,y) = u(t,x,y) :=  (1+\lambda)^{-1} \ind_{t < x < g(t,y)} = (1+\lambda)^{-1} \ind_{\tau(x,y) < t < x} .
\end{equation}
See Figure~\ref{fig:BBHM}.

\begin{figure}
\begin{center}
\includegraphics[height=3.5in]{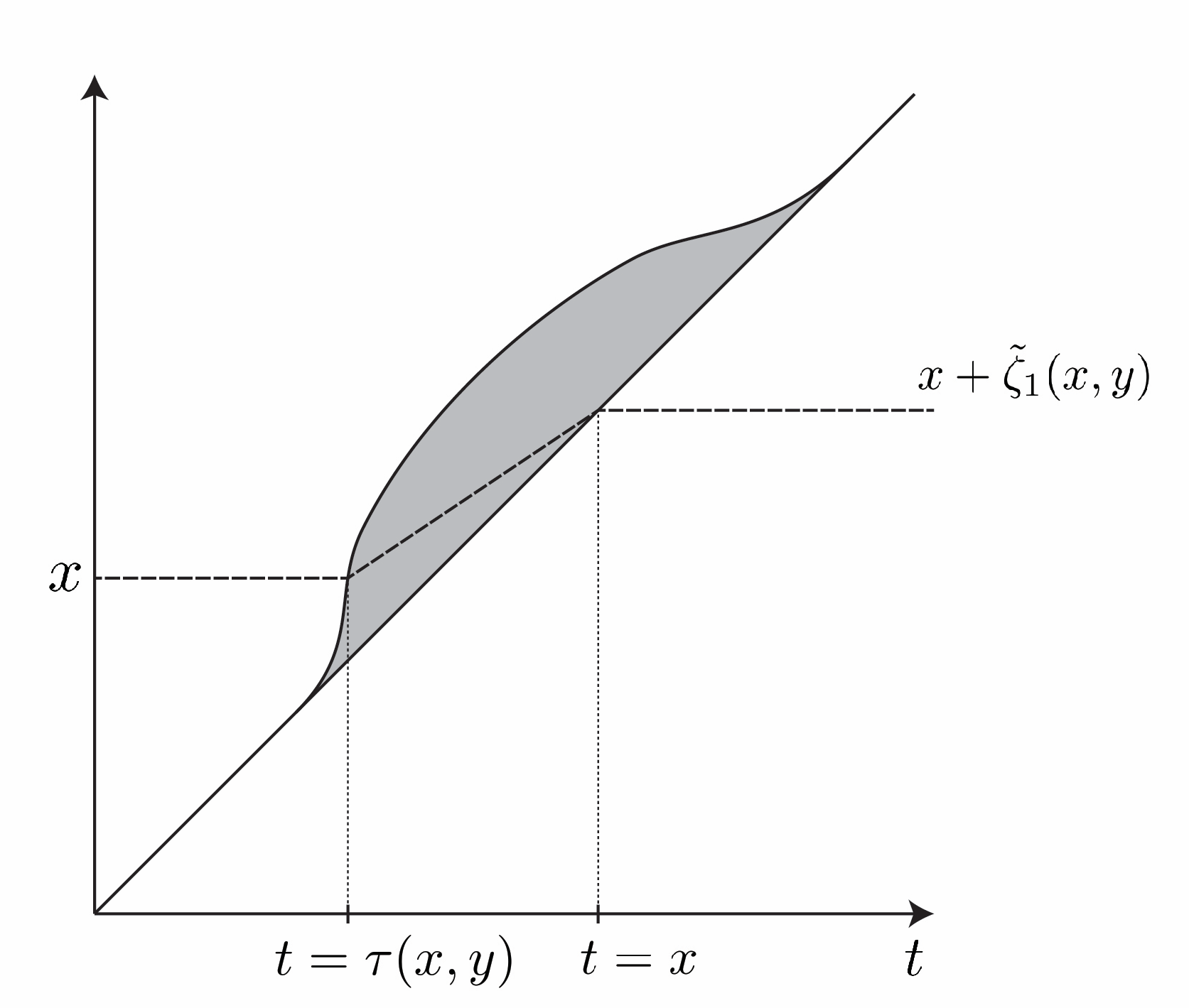}
\end{center}
\caption{A sketch of the flow $\ttheta$. 
The dashed line shows the trajectory starting from a point $x$ over time.
Its slope between $t=\tau(x,y)$ and $t=x$ is $(1+\lambda)^{-1}$.
The grey domain is the support of the vector field $u$.} 
\label{fig:BBHM}
\end{figure}

We will see below, in Lemma~\ref{lem:properties_g}, that $g(t,y) = t + \lambda\tzeta_1(t,y) + O(\lambda^2)$.
Since for every fixed $x$, $\tzeta_1(x,\cdot)$ is supported on $\approx k$ intervals of thickness $\approx \lambda$, 
it follows from \eqref{eq:def_u} and \eqref{eq:bounds_g} that for every fixed $t$, $u_t$ is supported on $\approx k$ disjoint compact sets, each contained in a square of edge length $\approx \lambda$, see Figure~\ref{fig:supp_u}.

\begin{figure}
\begin{center}
\includegraphics[height=3.5in]{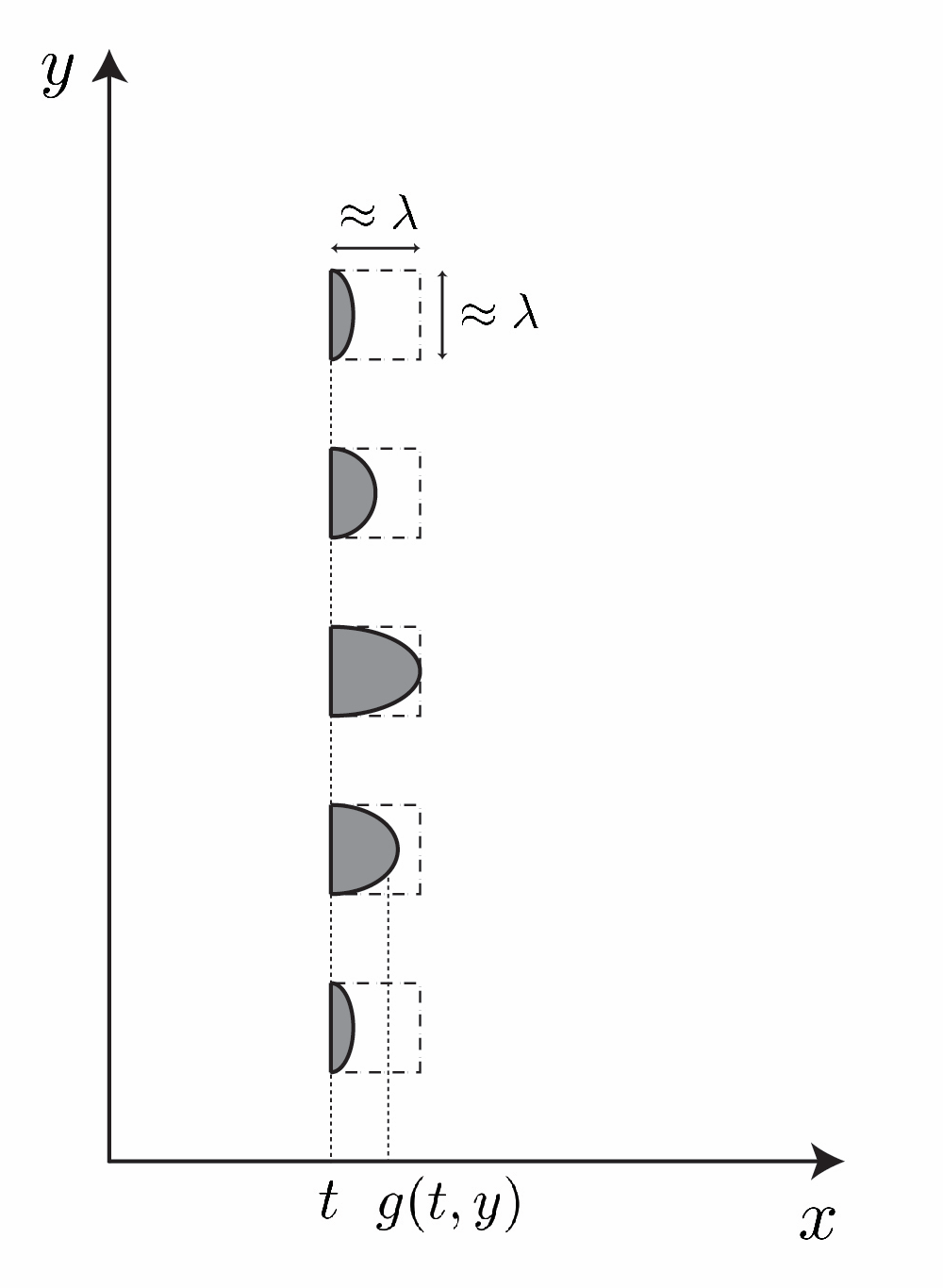}
\end{center}
\caption{A sketch of the support of $u_t$ for a fixed $t$.
The support consists of $\approx k$ sets, each contained in a square of diameter $\approx \lambda$.
Since the derivatives of $g$ are uniformly bounded \eqref{eq:bounds_dg}, the boundary of the support consists of $\approx k$ sets of length $\approx \lambda$.} 
\label{fig:supp_u}
\end{figure}

We obtained that $u_t$ has a small support, which is essential for using the subcriticality of $H^s$.
However, since $u_t \notin H^s$  for $s\ge 1/2$, we first need to regularize.
To do this, fix $\delta \ll \lambda$ (to be determined)
and define
\beq
\label{eq:u_delta_def}
u_{\delta,t}(x,y) = u_\delta(t,x,y) := \int_\R u(t,x-x',y) \eta_\delta(x')dx' =  \frac 1{1+\lambda}\int_{x-g(t,y)}^{x-t} \eta_\delta(x')dx'
\eeq
for $\eta_\delta\in C^\infty_c(\R)$ such that
\[
\eta_\delta\ge 0, \qquad\int_{-\infty}^0\eta_\delta = \int_0^\infty \eta_\delta= \frac{1}{2},\qquad
\mbox{supp}(\eta_\delta)\subset [-\delta, \delta], 
\qquad \|\eta_\delta\|_\infty \le \frac C {\delta}.
\]

Let $\theta(t,x,y)$ be the solution of
\beq
\label{eq:theta_def}
\frac \pl{\pl t}\theta(t,x,y) = u_\delta(t,\theta(t,x,y),y), \qquad \theta(0,x,y)=x
\eeq
and define $\theta(x,y) = \theta(1,x,y)$.
Define $\Theta\in \Diffc(\R^2)$ by 
\beq
\label{eq:Theta_def}
\Theta(x,y) = (\theta(x,y),y).
\eeq

In the rest of this section (which is by far the most technical part of this paper), we prove some estimates on $\Theta$. 
First, we prove that the path between $\id$ and $\Theta$ defined by flowing along $u_\delta$ is short, and therefore the distance from $\id$ to $\Theta$ is small (for an appropriate choice of $\lambda$ and $\delta$):
\begin{lemma}
\label{lem:Theta_cost}
\beq
\label{eq:Theta_cost}
\dist_s(\id,\Theta) \lesssim \frac{k^{1/2}\lambda^{(2-s)/2}}{\delta^{s/2}}
\eeq
\end{lemma}
The proof of this lemma will follow from Lemma~\ref{lem:bounds_u_delta} below.

We then prove that the regularization does not change the endpoint $\Theta$ by much (with respect to $\tTheta$), and we prove bounds on the derivatives of $\Theta$. These are concluded in the following proposition:
\begin{proposition}
\label{pn:Theta}
The diffeomorphism $\Psi^{-1}\circ \Theta \circ \Psi$ is of the form
\beq
\label{eq:Psi_Theta_Psi_form}
\Psi^{-1}\circ \Theta \circ \Psi = (x + \sigma(x,y), y),
\eeq
where $\sigma(x,y) \ge 0$ is supported on $(0,1)\times S_1$ and satisfies
\beq
\label{eq:bounds_sigma}
|\sigma(x,y) - \zeta_1(x,y)| \lesssim \frac{\delta}{\lambda}, \qquad
-1 + C^{-1} < \pl_x \sigma < C, \qquad
|\pl_y \sigma| \lesssim k.
\eeq
\end{proposition}
This proposition is proved at the end of this subsection, after some preliminary lemmas.
The conclusion of the proof of Theorem~\ref{thm:main_2D} (in Section~\ref{sec:Step_IV_2D} below) only uses \eqref{eq:Theta_cost}-\eqref{eq:bounds_sigma} and not the technical details that appear below in this subsection.

We begin the proofs of Lemma~\ref{lem:Theta_cost} and Proposition~\ref{pn:Theta} by some estimates on the unregularized flow $u$:
\begin{lemma}
\label{lem:properties_g}
The following bounds hold:
\beq
\label{eq:bounds_g}
g(t,y) = t + \lambda\tzeta_1(t,y) + O(\lambda^2), 
\qquad 
g(t,y) = t \iff \tzeta_1(t,y) = 0.
\eeq
\beq
\label{eq:bounds_dg}
\qquad \pl_1g = 1+ \lambda \pl_1\tzeta_1 + O(\lambda^2) = 1 +O(\lambda), 
\qquad |\pl_2g| < C.
\eeq
\end{lemma}

\begin{proof}
We fix $y$ and write $g(t) = g(t,y)$  and $\tzeta_1(t) = \tzeta_1(t,y)$.
Let $\tilde g(t) = t+\lambda\tzeta_1(t)$, and let $e(t) = g(t)-\tilde g(t)$.
Then
\[
t = \tau (g(t)) = \tau(t+\lambda \tzeta_1(t) + e(t)) = t+\lambda \tzeta_1(t) + e(t) -\lambda 
\tzeta_1\brk{t+\lambda \tzeta_1(t) + e(t)}.
\]
Thus $e = e(t)$ solves
\[
f(e;t) =  e +\lambda\tzeta_1(t) - \lambda \tzeta_1\brk{t+\lambda \tzeta_1(t) + e}  = 0.
\]
Since $|f(0;t)|\le \lambda^2 \|\pl_1\tzeta_1\|_\infty\|\tzeta_1\|_\infty <C\lambda^2$ for all $t$ and 
$\partial_e f \ge 1- \lambda \|\pl_1\tzeta_1\|_\infty \ge 1-C\lambda$ (here we use \eqref{eq:tilde_zeta_1_bounds}),
the Intermediate Value Theorem implies that a unique $e(t)$ such that $f(e(t);t)=0$ and $e(t)=O(\lambda^2)$.
The second part of \eqref{eq:bounds_g} is immediate from the definition of $g$.

For proving \eqref{eq:bounds_dg}, we use \eqref{eq:tilde_zeta_1_bounds} and calculate
\[
\pl_1 g = \pl_1 \tau ^{-1} = \frac 1{\pl_1 \tau\circ g} =
\frac 1{1-\lambda \pl_1 \tzeta_1\circ g} = 1+ \lambda \pl_1\tzeta_1 + O(\lambda^2),
\]
and
\[
\Abs{\pl_2 g} = \Abs{\frac{\pl_2 \tau}{\pl_1 \tau} }
= \Abs{ \frac{\lambda \pl_2 \tzeta_1}{1-\lambda \pl_1 \tzeta_1} } < C.
\]
\end{proof}

The following lemma, and in particular \eqref{eq:bounds_H_s_norm_u_delta}, immediately implies Lemma~\ref{lem:Theta_cost}.
\begin{lemma}
\label{lem:bounds_u_delta}
For a fixed $t$, $u_{\delta,t}(x,y)\in W^{1,\infty}(\R^2)$, and 
\beq
\label{eq:bounds_du_delta}
\|du_{\delta,t}\|_\infty \lesssim \frac{1}{\delta}.
\eeq
Moreover,
\beq
\label{eq:bounds_H_s_norm_u_delta}
\|u_{\delta,t}\|^2_{H^s} \lesssim \frac{k\lambda^{2-s}}{\delta^s}.
\eeq
\end{lemma}

\begin{proof}
$|\pl_1 u_{\delta,t}| < C/\delta$ follows from the definition of $u_\delta$ and the bounds on $\eta_\delta$.
We now show that $u_\delta$ is also Lipschitz with respect to the $y$ variable. 
Indeed, note that
\[
\Abs{u(t,x,y'+h) - u(t,x,y')} = (1+\lambda)^{-1} \ind_{g(t,y)< x < g(t,y+h)},
\]
if $g(t,y+h)>g(t,y)$, and similarly if not.
By \eqref{eq:bounds_dg},
\[
\Abs{g(t,y+h) - g(t,y)} \le |h|\,\|\pl_2 g\|_\infty \le C|h|
\]
and therefore we have
\[
\| u(t,\cdot,y'+h) - u(t,\cdot,y') \|_1 \le (1+\lambda)^{-1}C|h| \lesssim |h|.
\]
Finally,
\[
\Abs{u_\delta(t,x,y'+h) - u_\delta(t,x,y')} 
\le \|\eta_\delta\|_\infty \,\| u(t,\cdot,y'+h) - u(t,\cdot,y') \|_1
\lesssim \frac{|h|}{\delta},
\]
which completes the proof of \eqref{eq:bounds_du_delta}.

Now, similar to $u_t$, $u_{\delta,t}$ is supported on $\approx k$ disjoint compact sets, each contained in a square of edge length $\approx \lambda$.
Since $u_t$ is an indicator function, $du_{\delta,t}$ is supported on a $\delta$-neighborhood of the boundary of $\supp u_t$.
Since $|\pl_2 g| \le C$ (see \eqref{eq:bounds_dg}), it follows that $d u_{\delta,t}$ is supported on $\approx k$ sets of area of $\approx \delta\lambda$ (see Figure~\ref{fig:supp_u}).

Since $|u_{\delta,t}|_\infty < 1$, and $u_{\delta,t}$ is supported on a set of measure $\approx k\lambda^2$, we have
\[
\|u_\delta\|_2^2 \lesssim k\lambda^2.
\]
Since $|d u_{\delta,t}| \le C/\delta$, and $du_{\delta,t}$ is supported on a set of measure $\approx k\lambda\delta$,
\[
\|du_\delta,t\|_2^2 \lesssim \frac{k\lambda}{\delta}.
\]
Estimate \eqref{eq:bounds_H_s_norm_u_delta} follows from these bounds and Proposition~\ref{pn:GN_inequality}.
\end{proof}

Since we eventually want $u_{t,\lambda}$ to have a small $H^s$ norm, we will henceforth assume that $\delta$ satisfies
\beq
\label{eq:bounds_delta}
k\lambda^{2-s} \ll \delta^s \ll k^{-s^2/(1-s)}\lambda^s,
\eeq
where the upper-bound assumption (which is more restrictive than the natural $\delta\ll \lambda$) will be needed later.
In particular, note that these assumptions put some restrictions on the possible choices of $\lambda =  e^{-\alpha}/k$, in addition to \eqref{eq:bounds_alpha}.
We will give concrete choices of $\alpha$ and $\delta$ that satisfy these bounds in the end of the proof in Section~\ref{sec:Step_IV_2D}.

The following lemma states that the amount $\Theta$  "misses" the target $\tTheta$ because of the mollification is small:
\begin{lemma}
\label{lem:Theta_error}
$\supp(\theta(x,y) - x)$ is a subset of a $\delta$-thickening in the $x$ direction of $\supp(\tzeta_1)$, that is
\beq
\label{eq:theta_supp}
\supp(\theta(x,y) - x) \subset \BRK{(x,y)\, :\, \exists (x',y)\in \supp(\tzeta_1), \,\, |x-x'| < \delta}.
\eeq
In particular, for small enough $\delta$, $\supp(\theta(x,y) - x)\subset (0,1)^2$.
Moreover,
\beq
\label{eq:theta_error}
|\theta(x,y) - (x+ \tzeta_1(x,y))| \le 3\frac{\delta}{\lambda}
\eeq
\end{lemma}

\begin{proof}
Throughout this proof $y$ is fixed and does not play a role, and we will omit it for notational brevity.
Conclusion \eqref{eq:theta_supp} follows immediately from the definition of $\theta$. 
We now prove \eqref{eq:theta_error}.
Define
\begin{align*}
u_\delta^- &=  (1+\lambda)^{-1} \ind_{\BRK{u_\delta = (1+\lambda)^{-1}}} = (1+\lambda)^{-1} \ind_{t+\delta < x < g(t)-\delta}\ ,
\\ 
u_\delta^+ &= (1+\lambda)^{-1} \frac{1}{2}\brk{\ind_{\supp u} + \ind_{\supp u_\delta}} = (1+\lambda)^{-1} \brk{\ind_{t < x < g(t)} + \frac{1}{2} \ind_{\supp u_\delta \setminus \supp u}} 
\end{align*}
and let $\theta^\pm(t,x)$ solve
\[
\frac \pl{\pl t}\theta^\pm(t,x) = u_\delta^\pm(t,\theta^\pm(t,x)), \qquad \theta^\pm(0,x)=x.
\]
and let $\theta^\pm(x) := \theta^\pm (1, x)$.

It is clear that
\[
u_\delta^- \le u_\delta \le u_\delta^+
\]
pointwise. It follows that
$\theta^-(t,x)\le \theta(t,x)\le \theta^+(t,x)$
for all $t\ge 0$ and all $x$, and in particular $\theta^-(x)\le \theta(x)\le \theta^+(x)$.
See Figure~\ref{fig:u_delta_plus}.

\begin{figure}
\begin{center}
\includegraphics[height=3.5in]{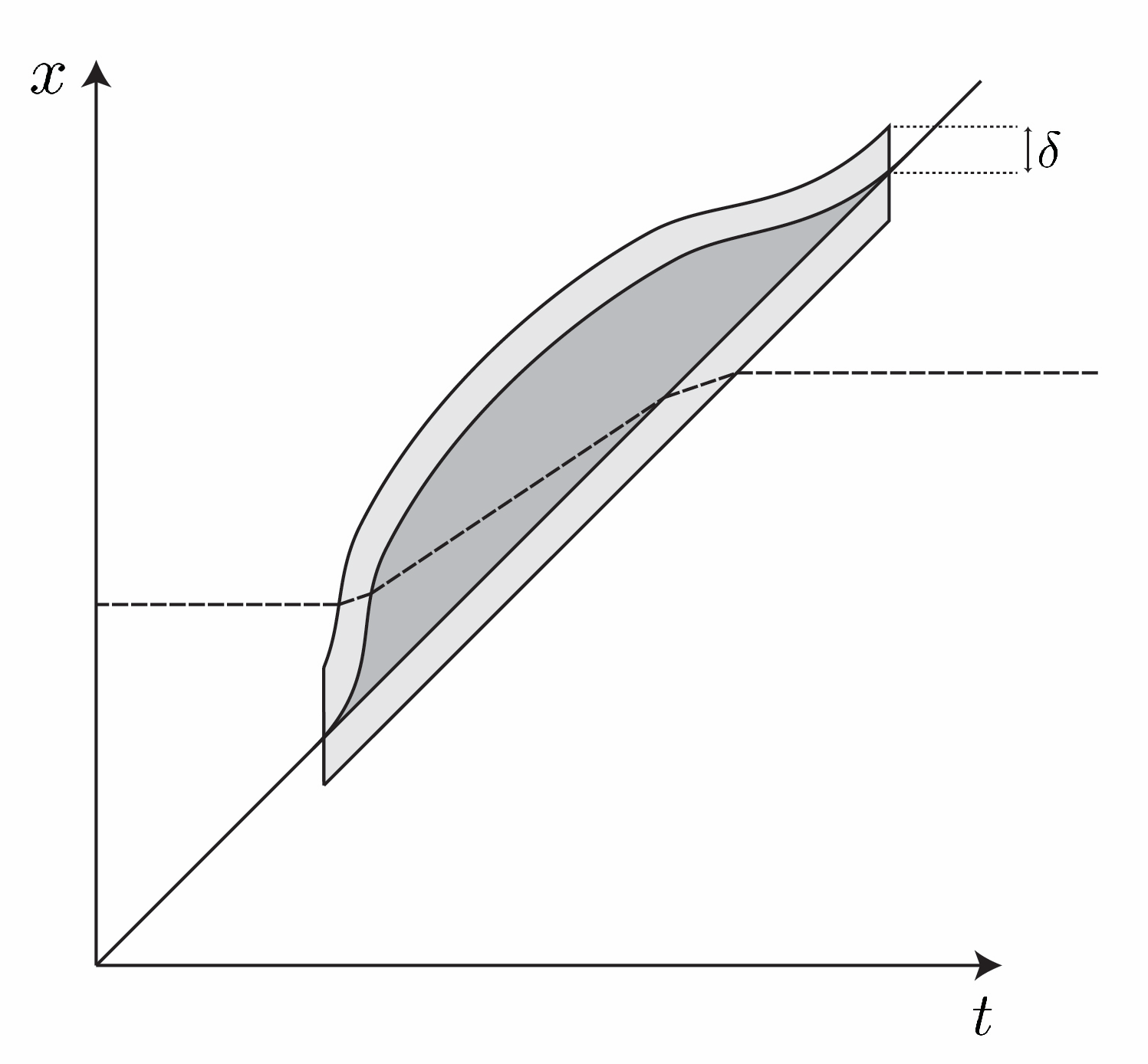}
\end{center}
\caption{A sketch of the flow $\theta^+$ along $u_\delta^+$.
The dark grey area is $\supp u$, where $u_\delta^+ = (1+\lambda)^{-1}$.
The light grey area is $\supp u_\delta\setminus \supp u$, which is at most of width $\delta$; in this region $u_\delta^+ = \frac{1}{2}(1+\lambda)^{-1}$. } 
\label{fig:u_delta_plus}
\end{figure}

First consider $\theta^+(t,x)$.Note that
$\theta^+(t,x) = x$ for $t\le t_1$, where $t_1$ is the first time such that $(t_1,x) \in \supp u_\delta$. 
Since $\supp \eta \subset [-\delta,\delta]$ we have
\[
t_1 \ge \tau(x-\delta).
\]
Since $\pl_1 \tau = 1 + O(\lambda)$ (see \eqref{eq:tilde_zeta_1_bounds}--\eqref{eq:tau_g_def}), it follows that $t_1 \ge \tau(x) - 2\delta$.
From $t_1$, until time $t_2$ defined by 
\[
g(t_2) = \theta^+(t_2,x),
\]
i.e.~the first time such that $(t_2,\theta^+(t_2,x))\in \supp u$, we have $\theta^+(t,x) < x + \frac{1}{2}(t-t_1)$ (note that for certain values of $x$, $(t,\theta^+(t,x))\notin \supp u$ for any $t$. In this case the analysis is simpler).
Using this inequality, \eqref{eq:bounds_dg} and the bound on $t_1$, it follows that $t_2 - t_1 \le 5\delta$.
Indeed, 
\[
x+\frac 12(t_2-t_1)>\theta^+(t_2,x)= g(t_2) > g(\tau(x)) +  (1-C\lambda)(t_2-\tau(x))
\]
and since $g(\tau(x)) = x$, we see that $\frac 12(t_2-t_1)> (1-C\lambda)(t_2-t_1 - 2\delta)$, from which the claim follows. 
Therefore $\theta^+(t_2,x) < x + 3\delta$.
Until the time $t_3$ when $\theta^+(t,x)$ leaves $\supp u$, $\theta^+$ flows according to the flow of $u$ with initial condition $\theta^+(t_2,x)$.
Therefore,
\[
\theta^+(t_3,x) = \theta^+(t_2,x) + \tzeta_1(\theta^+(t_2,x)) < x + \tzeta_1(x)+ C\delta,
\]
where we used \eqref{eq:tilde_zeta_1_bounds} again.
By the same arguments as for the time interval $[t_1,t_2]$, it follows that for $t>t_3$, $\theta^+(t,x)$ increases by less than $\delta$.
Therefore we obtain the upper bound
\beq
\label{eq:theta_error_upper_bound}
\theta(x) \le \theta^+(x) < x + \tzeta_1(x)+ C\delta,
\eeq
for an appropriate constant $C$.

We now consider $u_\delta^-$ and $\theta^-(t,x)$. 
Note that
\[
u_\delta^-(t,x) =  (1+\lambda)^{-1}\ind_{\tau(x+\delta)<t<x-\delta} > (1+\lambda)^{-1}\ind_{\tau(x) + 2\delta<t<x-\delta},
\]
where we used $\pl_1 \tau = 1 + O(\lambda)$ in the inequality.
Defining $t' = t +\delta$, we have 
\beq
\label{eq:u_delta_v_delta}
u_\delta^-(t',x) \ge v_\delta^-(t',x) := (1+\lambda)^{-1}\ind_{\max\BRK{\tau(x) + 3\delta,x}<t'<x}.
\eeq
By definition \eqref{eq:tau_g_def} of $\tau$
\[
\tau(x) + 3\delta = x - \lambda\tzeta_1(x) + 3\delta = x - \lambda\brk{\tzeta_1(x) - 3\frac{\delta}{\lambda}}.
\]
It follows that the flow by $v_\delta^-(t,x)$, that is the solution $\bar{\theta}^-$ of 
\[
\frac{\pl}{\pl t} \bar\theta^-(t,x) = v_\delta^-(t,\bar\theta^-(t,x))
\qquad \bar\theta^-(0,x)=x,
\]
satisfies
\[
\bar\theta^-(1,x) = \max\BRK{x+\tzeta_1(x) - 3\frac{\delta}{\lambda}, x}.
\]
Moreover, for $\delta$ small enough (depending only on $\zeta$), $\bar\theta^-(1-\delta,x) = \bar\theta^-(1,x)$.
By \eqref{eq:u_delta_v_delta}, it follows that
\beq
\label{eq:theta_error_lower_bound}
\theta(x) \ge \theta^-(1,x) \ge \bar\theta^-(1-\delta,x) \ge x+\tzeta_1(x) - 3\frac{\delta}{\lambda}.
\eeq
\eqref{eq:theta_error_upper_bound} and \eqref{eq:theta_error_lower_bound} imply \eqref{eq:theta_error}.
\end{proof}

Next, we prove bounds on the derivatives of $\theta$.
\begin{lemma}
\label{lem:theta_x_bounds}
There exists $C\ge 1$, depending only on $\zeta$, such that
\beq
\label{eq:bound_pl_x_theta}
C^{-1} \le \pl_x \theta \le C \qquad 
\text{for all $(x,y)$.}
\eeq
\end{lemma}

\begin{proof}
As in the proof of Lemma \ref{lem:Theta_error}, we will omit $y$ for notational brevity, and because it does not play any role.
Recall that $\pl_t \theta(t,x) = u_\delta(t,\theta)$, and consider the Eulerian version of this flow, that is the equation
\begin{equation}
\partial_t w(t,x) + u_\delta(t,x) \partial_xw(t,x) = 0
\label{eq:transport}\end{equation}
with initial data
\begin{equation}\label{eq:wzero}
w(0,x)= x.
\end{equation}
If $w$ is a solution then
\[
\frac d{dt}w(t,\theta(t,x)) = \partial_x w(t,\theta)\partial_t\theta
+\partial_t w(t,\theta)
= 0,
\]
using the ODE for $\theta$ and the PDE for $w$.
The initial data then imply that 
$w(t, \theta(t,x)) = x$ for all $t$,
and hence that 
\[
w(t,\cdot) = \theta(t,\cdot)^{-1}.
\]

Next, define
\[
q = \partial_t w+\pl_x w.
\]
Since $u_\delta(t,x) = 0$ when $t$ is close to $0$ or $1$, we have that $\partial_t w=0$ for such values of $t$. 
In particular, $q(0,\cdot) = 1$ and $q(1,\cdot)=\pl_x w(1,\cdot)=\pl_x \theta(1,\cdot)^{-1}$, which is the quantity we need to estimate.

We use $q$ and not $\pl_x w$ directly since it will allow us to exploit the fact,
reflected in the smallness of $(\partial_t+\partial_x)u_\delta$, that
the coefficients in \eqref{eq:transport} are nearly translation-invariant in the
$\partial_t+\partial_x$ direction. We compute
\begin{align*}
\pl_t q = \partial_t(\partial_t w+\pl_x w) =  (\partial_t+\partial_x)\partial_t w
&=
-(\partial_t+\partial_x)(u_\delta \pl_x w)=
-u_\delta \pl_x q - (\pl_tu_\delta+\pl_x u_\delta)\pl_x w.
\end{align*}
We further deduce from \eqref{eq:transport} that
\[
\pl_x w = q+ u_\delta \pl_x w,\qquad\mbox{ and thus }\qquad \pl_x w = \frac q{1-u_\delta},
\]
so we can rewrite the above equation as
\[
\pl_t q = -u_\delta \pl_x q - \frac{\pl_t u_\delta+\pl_x u_\delta}{1-u_\delta} q.
\]
It follows that
\begin{equation}\label{eq:growth}
\frac d{dt} q(t,\theta(t,x)) =- \frac{\pl_t u_\delta+\pl_x u_\delta}{1-u_\delta}\big(t,\theta(t,x)\big) \  q(t,\theta(t,x)).
\end{equation}
Therefore, if we obtain a bound
\beq
\label{eq:gronwall_estimate}
\int_0^1 \Abs{\frac{\pl_t u_\delta+\pl_x u_\delta}{1-u_\delta}\big(t,\theta(t,x)\big) }\,dt < C,
\eeq
for some $C$ independent of $x$ (and $y$), we obtain \eqref{eq:bound_pl_x_theta} by Gronwall's inequality.

From definition \eqref{eq:u_delta_def} of $u_\delta$,
we have
\begin{align}
\label{eq:u_t_and_u_x}
\partial_x u_\delta(t,x) 
&= \frac 1{1+\lambda} \Brk{\eta_\delta(x-t) - \eta_\delta(x-g(t))},
\\
\partial_t u_\delta(t,x) 
&= \frac 1{1+\lambda} \Brk{ -\eta_\delta(x-t) + g'(t) \eta_\delta(x-g(t))},
\end{align}
and therefore, using \eqref{eq:bounds_dg}, we have
\begin{equation}\label{eq:u_t_plus_u_x}
\begin{split}
\Abs{\partial_t u_\delta +\partial_x u_\delta }
&=  \frac 1{1+\lambda} \eta_\delta(x-g(t))\ \Abs{g'(t)-1}\\
&\le   \frac {C\lambda}{1+\lambda} \eta_\delta(x-g(t)).
\end{split}
\end{equation}

Because of \eqref{eq:growth} and \eqref{eq:u_t_plus_u_x}, we want to estimate $\frac{\eta_\delta(x-g(t))}{1-u_\delta(t,x)}$.
We have 
\begin{align*}
1 - u_\delta(t,x) 
&=1 -  \frac 1{1+\lambda}\int_{x-g(t)}^{x-t} \eta_\delta(x')dx'
\\
&\ge
1 -  \frac 1{1+\lambda}\int_{x-g(t)}^{\infty} \eta_\delta(x')dx'\\
&=
1 - \frac 1{1+\lambda} \mu_\delta (x-g(t)),
\qquad 
\mbox{ for } \ \ \mu_\delta(x) := \int_x^{\infty} \eta_\delta(x') dx' ,
\end{align*}
and therefore
\[ 
\frac{\eta_\delta(x-g(t))}{1-u_\delta(t,x)} 
\le 
\frac{(1+\lambda) \eta_\delta(x-g(t))}
{1+ \lambda - \mu_\delta(x-g(t))} 
=
-\frac{(1+\lambda) \mu_\delta'(x-g(t))}
{1+ \lambda - \mu_\delta(x-g(t))} .
\] 
It follows that 
\beq\label{eq:dtud1}
\int_0^1 \Abs{\frac{\pl_t u_\delta+\pl_x u_\delta}{1-u_\delta}\big(t,\theta(t,x)\big) }\,dt 
\le 
C\lambda
 \int_0^1
\frac{- \mu_\delta'(\theta(t,x)-g(t))}
{1+\lambda - \mu_\delta(\theta(t,x)-g(t))} 
 \, dt.
\eeq
For the following computation, $x$ is fixed. 
We wish to rewrite the integral in terms of the variable
\[
\alpha = \alpha(t) = \mu_\delta(\theta(t,x)-g(t)),
\]
which increases from $0$ to $1$ as $t$ goes from $0$ to $1$ for
$\delta, \lambda$ sufficiently small.  
To estimate $\alpha'(t)$, note that by the definition of $\theta$, we  have
\begin{align*}
\partial_t \theta (t,x)= u_\delta(t,\theta(t,x)) 
=
 \frac 1{1+\lambda}\int_{\theta(t,x)-g(t)}^{\theta(t,x)-t} \eta_\delta(x')dx'
&\le \frac{1}{1+\lambda} \int_{\theta(t,x)-g(t)}^\infty \eta_\delta(x')\,dx' \\
&=\frac{\alpha(t)}{1+\lambda}.
\end{align*}

Since $g' \ge 1-c\lambda$ for some $c<1$, depending only on $\zeta$, it follows that
\begin{align*}
\partial_t(\theta(t,x) - g(t))
&\le 
\frac{\alpha(t)}{1+\lambda} - 1 + c\lambda = \frac {  \alpha(t)   -(1+\lambda)(1-c\lambda)
}
{1+\lambda} .
\end{align*}
This is always negative for small enough $\lambda$, as $0\le\alpha\le 1$ and $c<1$.
Thus
\[
-\mu_\delta'(\theta(t,x)-g(t)) = \frac{\alpha'(t)}{ - \partial_t(\theta(x,t)-g(t))}
  \le \frac{(1+\lambda)\alpha'(t)}{ (1+\lambda)(1-c\lambda) - \alpha(t)}.
\]
So we can change variables in \eqref{eq:dtud1}
to find that
\begin{align*}
\int_0^1 \Abs{\frac{\pl_t u_\delta+\pl_x u_\delta}{1-u_\delta}\big(t,\theta(t,x)\big) }\,dt 
&\le
C\lambda \int_0^1 \frac 1{1+ \lambda-\alpha} \ \frac{1 }{(1+\lambda)(1-c\lambda) -\alpha}\  d\alpha.
\\
\end{align*}
For $\lambda < \frac {1-c}{2c}$, the integrand on the right is bounded by 
$(1+\frac 12(1-c)\lambda -\alpha)^{-2}$, so we integrate to conclude that
\[
\int_0^1 \Abs{\frac{\pl_t u_\delta+\pl_x u_\delta}{1-u_\delta}\big(t,\theta(t,x)\big) }\,dt 
\le
C\lambda\brk{\frac 12(1-c)\lambda}^{-1} \le C.
\]
We thus obtain \eqref{eq:gronwall_estimate}, which completes the proof.
\end{proof}

\begin{lemma}
\label{lem:theta_y_bounds}
For every $\lambda>0$ small enough, there exists a choice of mollifier $\eta_\delta$ in definition
\eqref{eq:u_delta_def} such that
\beq
\label{eq:bound_pl_y_theta}
|\pl_y \theta| \le C\lambda^{-1} \qquad 
\text{for all $(x,y)$},
\eeq
where $C>0$ depends only on $\zeta$.
\end{lemma}

\begin{proof}
Fix $h\in \R$, $|h|\ll\delta\lambda$, and consider $\theta(t,x,y)$ and $\theta(t,x,y+h)$.
By Lemma~\ref{lem:properties_g} we have that
\[
\Abs{\tau(t,y+h) - \tau (t,y)} < c|h|,
\]
for some $c>0$.
In particular, 
\[
\begin{split}
u(t-c|h|,x,y+h) &= (1+\lambda)^{-1} \ind_{\tau(x,y+h) < t - c|h| < x} \\
	&\le (1+\lambda)^{-1} \ind_{\tau(x,y) - c|h| < t -c|h|< x} = (1+\lambda)^{-1} \ind_{\tau(x,y) < t < x + c|h|} \\
	&= (1+\lambda)^{-1} \ind_{t-c|h|<x<g(t,y) } =: u^h(t,x,y).
\end{split}
\]
Therefore $u_\delta(t,x,y+h) \le u^h_\delta(t + c|h|,x,y)$, where $u^h_\delta$ is the mollification of $u^h$ as in \eqref{eq:u_delta_def}.
Define $\theta^h(t,x,y)$ by
\[
\frac \pl{\pl t}\theta^h(t,x,y) = u^h_\delta(t,\theta(t,x,y),y), \qquad \theta(0,x)=x.
\]
It follows that
\[
\theta(t - c|h|,x,y+h) \le \theta^h(t,x,y),
\]
and since for $h$ small enough (independent of $x$ and $y$), $\theta(1,x,y+h) = \theta(1-c|h|,x,y+h)$, we have
\[
\theta(1,x,y+h) \le \theta^h(1,x,y).
\]
We now compare $\theta^h(t,x,y)$ and $\theta(t,x,y)$ and show that
\beq
\label{eq:theta_h_minus_theta}
\theta(1,x,y+h) - \theta(1,x,y) \le \theta^h(1,x,y) - \theta(1,x,y) \lesssim \frac{|h|}{\lambda}.
\eeq
By symmetry it also follows that
\[
\theta(1,x,y) - \theta(1,x,y+h ) \lesssim \frac{|h|}{\lambda},
\]
which completes the proof.

It remains to prove the righthand side inequality in \eqref{eq:theta_h_minus_theta}.
In order to simplify notation, we will henceforth write $\theta(t) = \theta(t,x,y)$, $g(t) = g(t,y)$ and so on. 

For this, it is convenient to use a smooth mollifier $\eta_\delta$ with support in $[-\delta, \delta]$ such that
\[
0\le \eta_\delta(x) \le \frac {1+\lambda}{2\delta}.
\]
This is necessarily very close to the normalized characteristic function of the
interval $[-\delta, \delta]$ in $L^p$ for every $p<\infty$.
By the definition of $\theta^h(t)$, it follows (see Figure~\ref{fig:y_derivatives}) that 
\[
\theta^h(t) = \theta(t).
\]
for every $t<t_0$, where $t_0$ is defined by
\[
\theta(t_0) = t_0 + \delta.
\]

\begin{figure}
\begin{center}
\includegraphics[height=3.5in]{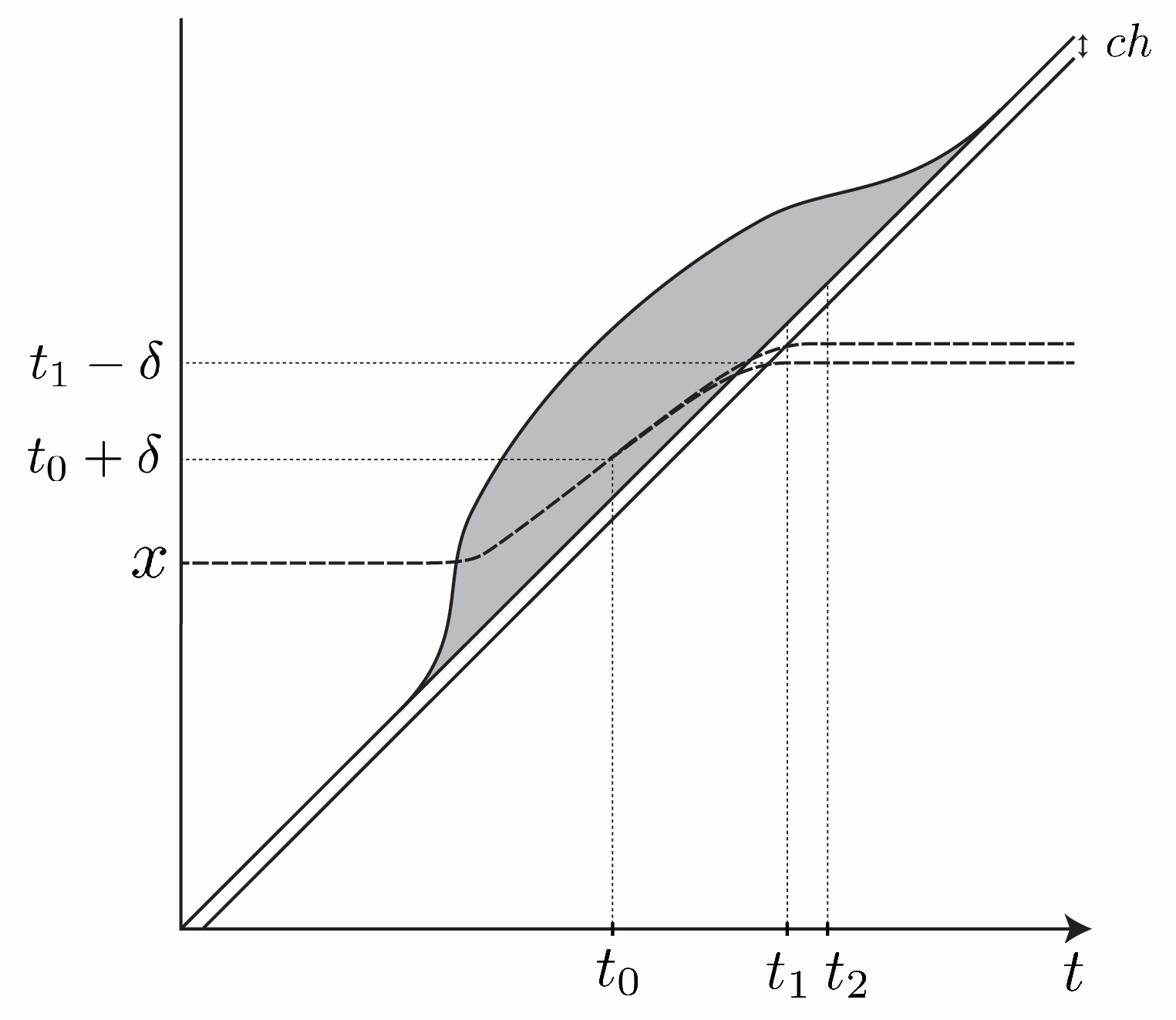}
\end{center}
\caption{A sketch of the trajectories of $\theta(t)$ (lower dashed line) and $\theta^h(t)$ (upper dashed line).
$\theta(t) = \theta^h(t)$ for $t\le t_0$.
$\theta(t)$ is constant after $t_1$ (where $\theta(t_1) = t_1-\delta$). 
$\theta^h(t)$ is constant after $t_2 = t_1 + \delta$, see \eqref{eq:theta_h_minus_theta_t_1}.} 
\label{fig:y_derivatives}
\end{figure}

When $\theta(t) -t \ge -\delta$, we have
\[
\frac d{dt}\theta = u_\delta(t,\theta) = \frac 1{1+\lambda}
\int_{\theta - g(t)}^{\theta-t}\eta(x')dx' \le 
\frac 1{1+\lambda} 
\int_{-\infty}^{\theta-t}\eta(x')dx'  
\le
\min \left\{ \frac {1}{2\delta}(\theta - t +\delta), \frac 1{1+\lambda} \right\},
\]
and when $\theta(t)-t\le -\delta$ we have $\frac{d\theta}{dt}=0$.
Let $\alpha(t) = \theta(t)-t$.
It follows that
\[
\frac {d\alpha}{dt} \le 
-\frac \lambda{1+\lambda}  
+
\min \left\{ \frac {1}{2\delta}( \alpha - \delta\frac{1-\lambda}{1+\lambda}), 0 \right\} \ \qquad\mbox{ as long as }\alpha(t)\ge -\delta,
\]
and $\frac{d\alpha}{dt}=-1$ when $\alpha(t)\le -\delta$.
If we write $\alpha_0(t)$ to denote the function solving the above ODE (with $\le$
replaced by $=$) with initial data $\alpha_0(t_0)=\delta$, then
$\alpha(t)\le \alpha_0(t)$ for $t\ge t_0$. This leads to
\[
\alpha(t)\le
\begin{cases}
\delta - \frac\lambda{1+\lambda}(t-t_0) &\mbox{ if } t_0 \le t \le t_a = t_0+ 2\delta
\\
\delta - \delta \frac{2\lambda}{1+\lambda}\exp(\frac{t-t_a}{2\delta}) &\mbox{ if }t_a\le t 
\end{cases}
\]
as long as $\alpha(t)\ge -\delta$. 

We now define $t_1$ to be the unique time such that $\alpha(t_1)=-\delta$,
and similarly $t_2$ such that $\alpha(t_2)=-2\delta$ (see Figure~\ref{fig:y_derivatives}). 
We deduce from the above that
\begin{equation}\label{tminust}
t_1 - t_0 \le 2\delta \left(
1 + \log(\frac{1+\lambda}{\lambda})\right),
\qquad t_2 - t_1 = \delta.
\end{equation}

Next we estimate $\theta^h(t_2)-\theta(t_2)$.
First note that
\[
0 \le u^h_\delta(t,x) - u_\delta(t,x)  = \frac 1{1+\lambda} \int_{x-t}^{x-t+c|h|}\eta_\delta(x')dx' \le \frac {c|h|}{2\delta}.
\]
We can similarly estimate $u_\delta(t,x') - u_\delta(t,x)$, to find that
\begin{align*}
\frac d{dt}(\theta^h - \theta) 
&=
[u^h_\delta(t, \theta^h) - u_\delta(t,\theta^h)] + [u_\delta(t,\theta^h) - u_\delta(t,\theta)]
\\
&\le \frac {c|h|}{2\delta} + \frac {1}{2\delta} (\theta^h-\theta).
\end{align*}
(We have implicitly used the fact that  $\theta^h(t)\ge \theta(t)$ for all $t$). Thus, Gr\"onwall's inequality implies that for $t>t_0$, 
\[
\theta^h(t) - \theta(t) \le  
 c|h| \left[ \exp\left(\frac{ t-t_0}{2\delta} 
\right)-1
\right].
\]
In particular, it follows from \eqref{tminust} that
\[
\theta^h(t_2)-\theta(t_2) \le
c|h| \left[ \exp\left(
\frac 32 + \log(\frac{1+\lambda}{\lambda})
 \right)-1
\right] \lesssim \frac {|h|}\lambda.
\]
Thus,
\beq
\label{eq:theta_h_minus_theta_t_1}
\theta^h(t_2) - t_2 = \theta^h(t_2) - \theta(t_2)  +\alpha(t_2)
	\le 
-2\delta+	c\frac{|h|}{\lambda} < -\delta
\eeq
for $h$ small enough.
Since $\theta^h(t) - t$ is a decreasing function, this inequality continues to hold after time $t_2$.
It then follows from the definitions that $u^h_\delta(t, \theta^h(t)) = u_\delta(t,\theta(t))= 0$
for $t\ge t_2$, and therefore
\[
\theta^h(1) - \theta(1) = \theta^h(t_2) - \theta(t_2) \lesssim \frac{|h|}{\lambda},
\]
which proves \eqref{eq:theta_h_minus_theta} and completes the proof.
\end{proof}

We conclude this section by completing the proof of Proposition~\ref{pn:Theta}:
{\flushleft \emph{Proof of Proposition~\ref{pn:Theta}}:}
The structure \eqref{eq:Psi_Theta_Psi_form} of $\Psi^{-1}\circ \Theta\circ \Psi$ is immediate from the definitions of $\Psi$ and $\Theta$.
We see from \eqref{eq:theta_supp}  that $\supp \sigma$ is a subset of a $\delta$-thickening in the $x$-direction of $\supp \zeta_1$.
Therefore, \eqref{eq:zeta_1_bounds_1} implies $\supp \sigma\subset (0,1)\times S_1$ for small enough $\delta$. 
The first bound in \eqref{eq:bounds_sigma} follows from Lemma~\ref{lem:Theta_error}, the second from Lemma~\ref{lem:theta_x_bounds}, and the third from Lemma~\ref{lem:theta_y_bounds}, using the fact that $\Psi$ is linearly squeezing strips on which $\theta$ is supported by a factor of $e^{-\alpha} = k\lambda$.
{\hfill\ding{110}}
\subsection{Step IV: Error correction --- affine homotopy}
\label{sec:Step_IV_2D}

In this subsection we correct the error obtained by the regularization in the previous subsection via affine homotopy, and then complete the proof.
The properties of the target of this affine homotopy, which follow from Proposition~\ref{pn:Theta}, are summed up in the following corollary:
\begin{corollary}
\label{cor:Gamma}
The diffeomorphism $\Gamma = \Psi^{-1}\circ \Theta \circ \Psi\circ \Phi_1^{-1}$ is of the form
\beq
\label{eq:Psi_Theta_Psi_form.b}
\Gamma = (\gamma(x,y), y) = (x + \xi(x,y), y),
\eeq
where $\xi(x,y) \ge 0$ is supported on $(0,1)\times S_1$ and satisfies
\beq
\label{eq:bounds_xi}
|\xi(x,y)| \lesssim \frac{\delta}{\lambda}, \qquad
-1 + C^{-1} < \pl_x \xi < C, \qquad
|\pl_y \xi| \lesssim k.
\eeq
\end{corollary}

\begin{proof}
This is immediate from Proposition~\ref{pn:Theta}, the definition of $\Phi_1$ and the bounds \eqref{eq:zeta_1_bounds_2}.
\end{proof}

\begin{lemma}
\label{lem:dist_s_Gamma_id}
\beq
\label{eq:dist_s_Gamma_id}
\dist_s(\Gamma,\id) \lesssim \frac{\delta^{1-s}}{\lambda^{1-s}}k^s.
\eeq
\end{lemma}

\begin{proof}
Consider an affine homotopy $\Gamma_t$ from $\id$ to $\Gamma$,
that is,
\[
\Gamma_t(x,y) = (x + t\xi(x,y),y) =: (\gamma_{t,y}(x),y)
\]
We then have $\pl_t \Gamma_t = u_t(\Gamma_t)$, where
\[
u_t(x,y) = (\xi(\Gamma_t^{-1}(x,y)),0) = (\xi(\gamma_{t,y}^{-1}(x),y),0).
\]
Note that $u_t$ is supported on a subset of the unit square, because $\xi$ is supported on a subset of the unit square and $\Gamma$ is a diffeomorphism of the unit square.
Since $|\xi|\lesssim \delta\lambda^{-1}$, we have
\beq
\label{eq:L_2_norm_affine}
\|u_t\|_{L^2} \lesssim \frac{\delta}{\lambda}.
\eeq
Next, we have 
\[
\pl_x u_t(x,y) = \pl_x \xi\, \pl_x \gamma_{t,y}^{-1}(x), 
\quad
\pl_y u_t(x,y) = \pl_x \xi\, \pl_y \gamma_{t,y}^{-1}(x) + \pl_y \xi.
\]
Since, by \eqref{eq:bounds_xi}, $-1+C^{-1}<\pl_x \xi<C$, we obtain that $|\pl_x \gamma_{t,y}^{-1}| = |1 + t\pl_x\xi|^{-1} <C$ and therefore $|\pl_x u|<C$.
Next, using \eqref{eq:bounds_xi} again, we have
\[
\Abs{\pl_y \gamma_{t,y}^{-1}(x)} \le \Abs{\frac{\pl_y \gamma_{t,y}}{\pl_x \gamma_{t,y}}} \lesssim k,
\]
and therefore $|\pl_y u_t| \lesssim k$. We conclude that
\beq
\label{eq:H_1_norm_affine}
\|u_t\|_{H^1} \lesssim k.
\eeq
Using Proposition~\ref{pn:GN_inequality}, \eqref{eq:L_2_norm_affine}--\eqref{eq:H_1_norm_affine} imply \eqref{eq:dist_s_Gamma_id}.
\end{proof}

We conclude now the proof of Theorem~\ref{thm:main_2D}.
We showed that
\[
\Phi_1 = \Gamma^{-1} \circ {\Psi}^{-1} \circ \Theta \circ \Psi, \qquad \Gamma,\Theta,\Psi \in \Diffc(\R^2),
\]
where (following Lemma~\ref{lem:squeezing_2D}, \eqref{eq:Theta_cost} and Lemma~\ref{lem:dist_s_Gamma_id})
\[
\dist_s(\Psi,\id) \lesssim \alpha k^{-(1-s)}, \qquad 
\dist_s(\Theta,\id) \lesssim \frac{k^{1/2}\lambda^{(2-s)/2}}{\delta^{s/2}}, \qquad 
\dist_s(\Gamma,\id) \lesssim \frac{\delta^{1-s}}{\lambda^{1-s}}k^s, \qquad \lambda = \frac{e^{-\alpha}}{k}.
\]
If we choose, say 
\[
\alpha = (\log k)^2, \qquad 
\lambda = \frac{1}{k^{1+\log k}}, \qquad 
\delta = \frac{1}{k^{\log k+\sqrt{\log k}}},
\]
we have, for any $s<1$,
\[
\begin{split}
\dist_s(\Psi,\id) &\lesssim (\log k)^2 k^{-(1-s)} = o(1), \\
\dist_s(\Theta,\id) &\lesssim k^{-(1-s)\log k + \frac{1}{2}s\sqrt{\log k}-\frac{1-s}{2}} = o(1), \\
\dist_s(\Gamma,\id) &\lesssim k^{1 -(1-s) \sqrt{\log k}} = o(1),
\end{split}
\]
and therefore $\dist_s(\Phi_1,\id) = o(1)$, which completes the proof.

\begin{remark}
Since we choose $\alpha$ and $\delta$ in an $s$-independent way, we constructed a sequence of paths from $\id$ to $\Phi$ that are of asymptotically vanishing $H^s$-cost for any $s<1$.
It follows that by choosing appropriate sequences of exponents $s_n \nearrow 1$ and constants $c_n \searrow 0$, we have
\[
\dist_{H^{<1}}(\Phi, \id) = 0,
\]
where the $H^{<1}$-norm is defined by
\[
\| f\|_{H^{<1}} := \sum_{n=1}^\infty c_n \| f\|_{H^{s_n}}.
\]
\end{remark}
\section{Higher-dimensional construction}
\label{sec:HD}
In this section we present a simpler construction in $\R^n$ for $n\ge 3$.
Since we often want to split $\R^n = \R\times \R^{n-1}$, it is convenient to
write $m=n-1$.

\begin{theorem}
\label{thm:main_HD}
Let $n\ge 3$, and denote by $(x,y)$ the coordinates on $\R^{n}$, where $x\in \R$ and $y\in \R^m$.
Let $\zeta \in C_c^\infty((0,1)^{n})$ satisfying $\zeta \ge 0$, $\pl_1\zeta > -1$. 
Denote $\phi(x,y) = x + \zeta(x,y)$.
Define $\Phi\in \Diffc(\R^{1+m})$ by $\Phi(x,y) = (\phi(x,y),y)$.
Then $\dist_s(\Phi,\id) = 0$ for every $s\in[0,1)$.
\end{theorem}

While in principle one can adjust the construction from the two-dimensional case to this setting, we can take advantage of the fact of the higher dimensionality to make a simpler construction, as outlined below: 
First, in Section~\ref{sec:Step_I_HD} we decompose $\Phi$ as follows:
\[
\Phi = \Phi_{2^m} \circ \ldots \circ \Phi_2\circ \Phi_1, \qquad \Phi_i = (\phi_i(x,y),y) = (x+\zeta_i(x,y),y) \in \Diffc(\R^{1+m}),
\]
where $\zeta_i$ is supported on the union of $\approx k^m$ "tubes" $(0,1)\times I_j$, where $I_j$ are $m$-dimensional cubes of edge length $\approx k^{-1}$. 
This is a generalization of the construction in Section~\ref{sec:Step_I_2D}.
In the rest of Section~\ref{sec:HD} we show that $\dist_s(\Phi_1,\id) = o(1)$ as $k\to \infty$, and the same holds for all the other $\Phi_i$s. 
Since $k$ is arbitrary, the conclusion $\dist_s(\Phi,\id) = 0$ follows by Lemma~\ref{lm:right_invariance}.

In order to prove $\dist_s(\Phi_1,\id) = o(1)$, we decompose $\Phi_1$ as
\[
\Phi_1 = \Psi^{-1} \circ \Gamma \circ \Psi, \qquad \Psi, \Gamma\in \Diffc(\R^{1+m}),
\]
where
\begin{enumerate}
\item $\Psi(x,y) = (x,\psi(x,y))$ squeezes the $m$-dimensional cubes $I_j$ on which $\Phi_1$ is supported by a factor of $k^{\log k}$.
In Section~\ref{sec:Step_II_HD}, we define $\Psi(x,y)$ and show that $\dist_s(\Phi,\id) \lesssim (\log k)^2 k^{-(1-s)} = o(1)$.
This is analogous to Section~\ref{sec:Step_II_2D}, with $\alpha = (\log k)^2$.
\item
$\Gamma = \Psi \circ \Phi_1 \circ \Psi^{-1}$. 
Unlike in the two-dimensional case, we do not have to construct a complicated flow along the strips (as in Section~\ref{sec:Step_III_2D}, which is the main part of the proof). This is because the squeezing in $m$-dimensions is enough to guarantee small norm, as explained in Section~\ref{sec:2dc}.
Instead, in Section~\ref{sec:Step_III_HD}, we show that the affine homotopy between $\id$ and $\Gamma$ is a path of small $H^s$ distance, and therefore $\dist_s(\Gamma,\id) \lesssim k^{s-(m/2 -s) \log k} = o(1)$.
\end{enumerate}
It then follows from Lemma~\ref{lm:right_invariance} that $\dist_s(\Phi_1,\id) = o(1)$.

\subsection{Step I: Splitting into strips}
\label{sec:Step_I_HD}

Fix $k\in \N$, and consider the lattice $\frac{4}{k}\Z^m \subset \R^m$.
We partition $\Z^m$ into $2^m$ latices: 
\[
2\Z^m,\, 2\Z^m + e_1,\, \ldots,\, 2\Z^m + \sum_{i=1}^m e_i,
\]
where $\{e_i\}_{i=1}^m$ is the standard basis of $\R^m$,
and similarly for the lattice $\frac{4}{k}\Z^m$.
We index the different lattices as $Z_I$, $I\in \Z_2^m$, ordered by 
\[
(0,\ldots,0), (1,0,\ldots,0),  (0,1,0,\ldots,0), \ldots, (0,1,1,\ldots,1), (1,\ldots,1).
\]
Sometimes we will denote the indices by $1,\ldots, 2^m$ according to this order.
For each $I\in \Z^m_2$, denote 
\[
L_I := \brk{Z_I + \Brk{-2/k,2/k}^m} \cap [0,1]^m, \qquad S_I := \brk{Z_I + \brk{-3/k,3/k}^m} \cap [0,1]^m.
\]
Note that $\cup L_I = [0,1]^m$ and that $L_I$ may only intersect $L_J$ at its boundary.

We now define diffeomorphisms $\Phi_I(x,y) = (x + \zeta_I(x,y),y)$, such that $\Phi = \Phi_{2^m} \circ\ldots \circ \Phi_{1}$,
\beq
\label{eq:zeta_1_bounds_0_HD}
\Phi_{I} \circ\ldots \circ \Phi_{1} |_{(0,1)\times \cup_{J\le I} L_J} = \Phi,
\eeq
\beq
\label{eq:zeta_1_bounds_1_HD}
\supp(\zeta_I) \subset (0,1)\times S_I,
\eeq
and
\beq
\label{eq:zeta_1_bounds_2_HD}
0\le \zeta_I\le C,\quad
-1 + C^{-1}< \pl_x \zeta_I < C, \quad
|\pl_y \zeta_I| < Ck,
\eeq
for some $C$ independent of $k$.

Let $\chi_I(y)$ be a bump function such that $\chi_I|_{L_I} \equiv 1$, $\supp\chi_I \subset S_I$ and $|d\chi_I| < Ck$.
Define 
\[
\zeta_1(x,y) = \zeta(x,y) \chi_1(y).
\]
For $I=2,\ldots, 2^m -1$, define
\[
\tilde{\Phi}_I := \Phi\circ \Phi_1^{-1} \circ \ldots \circ \Phi_{I-1}^{-1} = (x+ \tilde{\zeta}_I(x,y),y),
\]
and then
\[
\zeta_I(x,y) = \tilde{\zeta}_I(x,y) \chi_I(y).
\]
Finally, define
\[
\Phi_{2^m} := \Phi\circ \Phi_1^{-1} \circ \ldots \circ \Phi_{2^m-1}^{-1}.
\]
A direct calculation shows that $\Phi_I$ satisfies  \eqref{eq:zeta_1_bounds_0_HD}-\eqref{eq:zeta_1_bounds_2_HD}.

In the rest of this section we are going to prove that $\dist_s(\Phi_1,\id) = o(1)$.
This relies only on properties \eqref{eq:zeta_1_bounds_1_HD}--\eqref{eq:zeta_1_bounds_2_HD}, hence the result also applies to $\Phi_I$, for all $I\in \Z_2^m$, since $\zeta_I$ satisfies the same assumptions.

\subsection{Step II: Squeezing the strips}
\label{sec:Step_II_HD}

\begin{lemma}
\label{lem:squeezing_HD}
Fix $\alpha \gg 1$. 
There exists a diffeomorphism $\Psi\in \Diffc(\R^{1+m})$, $\Psi(x,y) = (x,\psi(x,y))$, such that
\beq
\label{eq:squeezing_HD}
\psi(x,y) = e^{-\alpha}\brk{y-z} + z,
\eeq
for every $x\in [0,1]$ and $y\in S_1$ such that $z\in \frac{8}{k}\Z^m$ is the closest element to $y$ in $\frac{8}{k}\Z^m$.
Moreover,
\beq
\label{eq:squeezing_HD_H_s_dist}
\dist_s(\Psi,\id) \lesssim \alpha k^{-(1-s)}. 
\eeq
\end{lemma}

\begin{proof}
Let $u_1 \in C_c^\infty((-4,4)^m)$, such that $u_1(y) = -y$ for $y\in [-3,3]^m$, and extend periodically to $\R^m$. 
Let $\chi \in C_c^\infty(\R^{1+m})$ such that $\chi\equiv 1$ on $[0,1]^{1+m}$.
Define $u_k(x,y) := \frac{\alpha}{k} u_1(ky)\chi(x,y)$.
The proof continues in the same way as the proof of Lemma~\ref{lem:squeezing_2D}.
\end{proof}

Note that in $[0,1]^{1+m}$, $\psi$ is independent of $x$. Therefore, slightly abusing notation, we write
\[
\Psi(x,y) = (x,\psi(y)), \qquad \Psi^{-1}(x,y) = (x,\psi^{-1}(y)).
\]
We will later have $\alpha$ depend on $k$.

\subsection{Step III: Affine homotopy}
\label{sec:Step_III_HD}

\begin{lemma}
\label{lem:dist_s_Gamma_id_HD}
\[
\dist_s(\Gamma,\id) \lesssim  k^{m/2}\lambda^{m/2-s} = k^s e^{-(m/2-s)\alpha}
\]
where $\Gamma = \Psi \circ \Phi_1 \circ \Psi^{-1}$ and $\lambda = e^{-\alpha}/k$.
\end{lemma}

\begin{proof}
Note that
\[
\Gamma = (x + \zeta_1(x,\psi^{-1}(y)),y), 
\]
and denote
\[
\xi(x,y) := \zeta_1(x,\psi^{-1}(y)), \qquad \gamma(x,y) = x + \zeta_1(x,\psi^{-1}(y)).
\]
It follows from the definitions of $\zeta_1$ \eqref{eq:zeta_1_bounds_1_HD} and $\psi$ \eqref{eq:squeezing_HD} that $\xi$ is supported inside $(0,1)\times \psi(S_1)$, i.e., inside $\approx k^m$ "tubes" which are translations of $(0,1) \times [-3\lambda,3\lambda]^m$.
In particular,
\beq
\label{eq:xi_support_HD}
\Vol(\supp \xi) \lesssim k^m\lambda^m.
\eeq
Furthermore, as in \eqref{eq:tilde_zeta_1_bounds}, we have from \eqref{eq:zeta_1_bounds_2_HD} that
\beq
\label{eq:xi_bounds_HD}
0\le \xi\le C,\quad
-1 + C^{-1}< \pl_x \xi < C, \quad
|\pl_y \xi| < C\lambda^{-1}.
\eeq
Consider now an affine homotopy $\Gamma_t$ from $\id$ to $\Gamma$,
that is,
\[
\Gamma_t(x,y) = (x + t\xi(x,y),y).
\]
The same calculation as in Lemma~\ref{lem:dist_s_Gamma_id}, using the estimates \eqref{eq:xi_support_HD}--\eqref{eq:xi_bounds_HD}, yields the wanted bound on $\dist_s(\id,\Gamma)$.
\end{proof}

We conclude now the proof of Theorem~\ref{thm:main_HD}.
We showed that
\[
\Phi_1 = \Psi^{-1} \circ \Gamma \circ \Psi,
\]
where (following Lemmas~\ref{lem:squeezing_HD}~\ref{lem:dist_s_Gamma_id_HD})
\[
\dist_s(\Psi,\id) \lesssim \alpha k^{-(1-s)}, \qquad 
\dist_s(\Gamma,\id) \lesssim  k^s e^{-(m/2-s)\alpha}.
\]
Recall that $m = n-1\ge 2$ by hypothesis. If we choose, say 
\[
\alpha = (\log k)^2,
\]
we have, for any $s<1$,
\[
\begin{split}
\dist_s(\Psi,\id) &\lesssim (\log k)^2 k^{-(1-s)} = o(1), \\
\dist_s(\Gamma,\id) &\lesssim k^{s-(m/2-s)\log k} = o(1),
\end{split}
\]
and therefore $\dist_s(\Phi_1,\id) = o(1)$, which completes the proof.

\section{The construction for $W^{s,p}(\R^n)$, $n \ge 2$.}\label{sec:Wsp}

In this section we explain how to modify the arguments presented above
in order to extend our earlier construction to the induced $W^{s,p}$
geodesic distance on $\Diffc(\R^n)$ for $n\ge 2$.

\begin{theorem}
\label{thm:main_HD2}
Let $n\ge 2$, and denote by $(x,y)$ the coordinates on $\R^{n}$, where $x\in \R$ and $y\in \R^m$ for $m = n-1$.
Let $\zeta \in C_c^\infty((0,1)^{n})$ satisfying $\zeta \ge 0$, $\pl_1\zeta > -1$. 
Denote $\phi(x,y) = x + \zeta(x,y)$.
Define $\Phi\in \Diffc(\R^{n})$ by $\Phi(x,y) = (\phi(x,y),y)$.
Then $\dist_{s,p}(\Phi,\id) = 0$ for every $s\in[0,1)$ and $p\ge 1$ such that 
$sp<n$.
\end{theorem}

As explained at the end of Section \ref{sec_2}, this will complete the proof of Theorem 
\ref{main_thm}.

We will use the interpolation inequality of Proposition \ref{pn:GN_inequality}
to estimate $W^{s,p}$-norms. This is not valid for $p=1$, but for functions
$u$ with compact support, it follows easily from the definition \eqref{def:fractional_Sobolev} and H\"older's inequality that
$\| u\|_{s,1} \le C(q,\supp(u)) \| u\|_{s,q}$ for every $q>1$, so the $p=1$ case
follows from estimating $\| u\|_{s,q}$ for $q>1$, for $q$ close enough to $1$ (in the construction below the vector fields are independent of the exponent).

\begin{proof}

{\bf 1. Splitting into strips and squeezing the strips}

Fix $k\in \N$.
We start exactly as in Section \ref{sec:Step_I_HD} by writing 
$\Phi = \Phi_{2^m} \circ\ldots \circ \Phi_{1}$,
where $\Phi_I$ satisfies \eqref{eq:zeta_1_bounds_1_HD}, \eqref{eq:zeta_1_bounds_2_HD} for $I = 1,\ldots, 2^m$.

It now suffices to show that $\dist_{s,p}(\id, \Phi_1) = o(1)$
as $k\to \infty$, at a rate that depends only on the  
constants in 
\eqref{eq:zeta_1_bounds_1_HD}, \eqref{eq:zeta_1_bounds_2_HD},
and that thus applies to $\Phi_2,\ldots, \Phi_{2^m}$ as well.

To do this, we start with the (higher-dimensional) squeezing diffeomorphism $\Psi$ from Lemma
\ref{lem:squeezing_HD}. Then the interpolation inequality from Proposition
\ref{pn:GN_inequality} yields
\begin{equation}\label{Psi.HD}
\dist_{s,p}(\Psi,\id) \lesssim \alpha k^{-(1-s)}\qquad\mbox{ for all $p\in (1,\infty)$}.
\end{equation}

{\bf 2. Flowing along the squeezed strips}.

We will now follow the procedure of Section \ref{sec:2dc} and  write
\begin{equation}\label{GD}
\Phi_1 = \Gamma^{-1} \circ {\Psi}^{-1} \circ \Theta \circ \Psi, \qquad \Gamma,\Theta,\Psi \in \Diffc(\R^2),
\end{equation}
where the construction of $\Theta, \Gamma$ and accompanying estimates closely follow the two-dimensional constructions in Sections \ref{sec:Step_III_2D} and  \ref{sec:Step_IV_2D}.

In more detail, to define $\Theta$, we first define $\ttheta(t,x,y)$  and $u(t,x,y)$
as in \eqref{eq:ttheta_def} and \eqref{eq:def_u}, with the only difference that now 
$y\in \R^{n-1}$. We then define $u_\delta$ as in
\eqref{eq:u_delta_def}, by convolving $u$ (in the $x$ variable only)
with a mollifier $\eta_\delta$.
Finally, we let $\theta(t,x,y)$ solve the ODE \eqref{eq:theta_def},
and we define
$\Theta(x,y) = (\theta(x,y,1), y)$.

Then Lemma~\ref{lem:properties_g} holds as is, and in
Lemma~\ref{lem:bounds_u_delta}, \eqref{eq:bounds_du_delta} holds and \eqref{eq:bounds_H_s_norm_u_delta} becomes
\[
\|u_{\delta,t}\|_{W^{s,p}}^p \lesssim \frac{k^{n-1} \lambda^{n-s}}{\delta^{(p-1)s}},
\]
and hence,
\[
\dist_{s,p}(\Psi,\id) \lesssim \frac{k^{(n-1)/p} \lambda^{(n-s)/p}}{\delta^{(p-1)s/p}}.
\]
{\bf 3. Error correction --- affine homotopy}

We define $\Gamma$ by \eqref{GD}, and we estimate $\dist_{s,p}(\id, \Gamma)$
by using an affine homotopy. 
Lemmas~\ref{lem:Theta_error}--\ref{lem:theta_y_bounds} hold as is, hence Proposition~\ref{pn:Theta} and Corollary~\ref{cor:Gamma} as well.
 Lemma~\ref{lem:dist_s_Gamma_id} holds as well, yielding
\[
\dist_{s,p}(\Gamma,\id) \lesssim \frac{\delta^{1-s}}{\gamma^{1-s}}k^s.
\]
The estimate is independent of $p\in (1,\infty)$ and $n$
as a consequence of the fact that the velocity field $u_t$, $0\le t\le 1$
associated to the affine homotopy (which in fact does not depend on $t$)
satisfies estimates that are uniform in $p$ and $n$. This follows from easy
modifications of the proofs of
\eqref{eq:L_2_norm_affine}, \eqref{eq:H_1_norm_affine}.
The constant in the above inequality does depend on $p$ through the dependence on the constant in the interpolation inequality.

{\bf 4. Conclusion of the proof}

Again, choosing 
\[
\alpha = (\log k)^2, \qquad \lambda = \frac{1}{k^{1+\log k}} \qquad \delta = \frac{1}{k^{\log k+\sqrt{\log k}}},
\]
we have, for any $s<\min\BRK{n/p,1}$,
\[
\begin{split}
\dist_{s,p}(\Psi,\id) &\lesssim (\log k)^2 k^{-(1-s)} = o(1), \\
\dist_{s,p}(\Theta,\id) &\lesssim k^{-\brk{\frac{n}{p}-s}\log k + \frac{p-1}{p}s\sqrt{\log k}-\frac{1-s}{p}} = o(1), \\
\dist_{s,p}(\Gamma,\id) &\lesssim k^{1 -(1-s) \sqrt{\log k}} = o(1),
\end{split}
\]
and therefore $\dist_{s,p}(\Phi_1,\id) = o(1)$.
\end{proof}

In the far subcritical regime $s<\min\BRK{(n-1)/p,1}$,
one can also give a simpler construction, like that of Section \ref{sec:HD},
in which the flow along the squeezed strips is carried out by an affine homotopy, and
no error-correction is needed at the end. 
Again, this is because the $(n-1)$-dimensional squeezing of the second step is enough to guarantee a small norm for the affine homotopy, since $W^{s,p}(\R^{n-1})$ is subcritical.
We do not think this has any deeper meaning besides the obvious observation that the weaker the norm is, the easier it is to construct paths of short length.

\paragraph{Acknowledgements}
We are grateful to Meital Kuchar for her help with the figures, and to the anonymous referee for their helpful comments.
This work was partially supported by the Natural Sciences and Engineering Research Council of Canada under operating grant 261955.


{\footnotesize
\bibliographystyle{amsalpha}
\providecommand{\bysame}{\leavevmode\hbox to3em{\hrulefill}\thinspace}
\providecommand{\MR}{\relax\ifhmode\unskip\space\fi MR }
\providecommand{\MRhref}[2]{%
  \href{http://www.ams.org/mathscinet-getitem?mr=#1}{#2}
}
\providecommand{\href}[2]{#2}

}
\end{document}